\documentclass[11pt,a4paper]{article}

\usepackage[utf8]{inputenc}
\usepackage[T1]{fontenc}
\usepackage{amsmath,amssymb,amsfonts,amsthm,mathtools}
\usepackage{geometry}
\usepackage{enumitem}
\usepackage{microtype}
\usepackage{hyperref}
\usepackage{cite}
\usepackage{xcolor}
\hypersetup{colorlinks=true,linkcolor=blue,citecolor=blue,urlcolor=blue}

\newtheorem{theorem}{Theorem}[section]
\newtheorem{proposition}[theorem]{Proposition}
\newtheorem{lemma}[theorem]{Lemma}
\newtheorem{corollary}[theorem]{Corollary}
\theoremstyle{definition}
\newtheorem{definition}[theorem]{Definition}
\newtheorem{remark}[theorem]{Remark}

\newcommand{\C}{\mathbb C}
\newcommand{\Der}{\operatorname{Der}}
\newcommand{\Inn}{\operatorname{Inn}}
\newcommand{\LocDer}{\operatorname{LocDer}}
\newcommand{\ad}{\operatorname{ad}}
\newcommand{\g}{\mathfrak g}
\newcommand{\slc}{\mathfrak{sl}_2(\C)}
\newcommand{\sod}{\mathfrak{so}_d(\C)}

\title{\textbf{Local Derivations on Conformal Galilei Algebras and Their Central Extensions}}
\author{
Khusainboy Atajonov\thanks{National University of Uzbekistan named after Mirzo Ulugbek, 4 University Street, 100174, Tashkent, Uzbekistan. Email: \texttt{atajonovxusainboy@gmail.com}}
\and
Bekzod B. Sultonboyev\thanks{Department of Algebra and Mathematical Engineering, Urgench State University, Urgench, Uzbekistan. Email: \texttt{bekzodbekbobomurodovich@gmail.com}}
\and
Bakhtiyor B. Yusupov\thanks{V.~I.~Romanovskiy Institute of Mathematics, Academy of Sciences of Uzbekistan, Tashkent, Uzbekistan; Department of Algebra and Mathematical Engineering, Urgench State University, Urgench, Uzbekistan. Email: \texttt{baxtiyor\_yusupov\_93@mail.ru}}
}
\date{}

\begin{document}
\maketitle

\begin{abstract}
We give a unified study of local derivations on conformal Galilei Lie algebras in three natural settings: the algebra without central extension, the mass central extension, and the exotic central extension. We first determine the structural form of the derivation algebras and isolate the relevant outer derivations. For the mass extension, we complete the one-string rigidity at the two low parameters not covered by the previous one-dimensional treatment, using direct calculations, and combine this with new spatial compatibility arguments to obtain the result in arbitrary spatial dimension. The exotic extension is treated by a self-contained argument from its defining relations. For the algebra without central extension we resolve the remaining algebraic-reflexivity problem completely. If $d\ge2$, every local derivation is again a derivation. In spatial dimension $d=1$, the same conclusion holds when $2\ell$ is even and also for $\ell=\frac12$; however, for odd $2\ell\ge3$ there is an additional $(2\ell-1)$-dimensional space of pure local derivations. This gives a sharp contrast with the mass central extension, where the central Heisenberg pairing removes precisely this extra local freedom. We also determine the Lie algebra structure of the full local-derivation space: in the exceptional one-dimensional non-central case it is an explicit semidirect product with an additional abelian irreducible ideal.
\end{abstract}


\section{Introduction}
The theory of local derivations studies linear mappings which agree, at each individual element, with a derivation that may depend on that element. The notion was introduced independently by Kadison and by Larson--Sourour \cite{Kadison,LarsonSourour}. If $L$ is a Lie algebra, a linear map $\Delta:L\to L$ is called a \emph{local derivation} if for every $x\in L$ there exists a derivation $D_x\in\Der L$ such that
\[
\Delta(x)=D_x(x).
\]
The central problem is to determine when this pointwise condition forces $\Delta$ itself to be a derivation. For Lie and Leibniz algebras the answer depends strongly on the structure of the radical, the module action of a Levi factor, and the presence of central extensions.

A fundamental rigidity result of Ayupov and Kudaybergenov states that every local derivation on a finite-dimensional semisimple Lie algebra over an algebraically closed field of characteristic zero is a derivation \cite{AK}. The corresponding two-point problem had already led Ayupov, Kudaybergenov and Rakhimov to the rigidity theorem for $2$-local derivations on finite-dimensional semisimple Lie algebras \cite{AKR}. Subsequent works of Ayupov, Kudaybergenov, Yusupov and their collaborators developed the local and $2$-local theory for a broad range of non-associative algebras. In particular, local and $2$-local derivations were investigated for $p$-filiform Leibniz algebras \cite{AKYp}, generalized Witt algebras \cite{AKYWitt}, locally simple Lie algebras \cite{AKYLocSimple}, and solvable Leibniz algebras \cite{AKHY}. For solvable Lie algebras, the maximal-rank case was treated by Kudaybergenov, Omirov and Kurbanbaev \cite{KOK}, while Yusupov and Yuldashev later obtained a local-derivation rigidity result for solvable Lie algebras with a filiform nilradical \cite{YYFil}. Together these results exhibit both local rigidity and the existence of pure local derivations, showing that non-semisimple structure alone does not determine the local--global behavior.

The local--global rigidity phenomenon has also been developed independently in several other Lie-theoretic directions. Yu and Chen proved that every local derivation on a standard Borel subalgebra of a finite-dimensional simple Lie algebra is a derivation \cite{YuChenBorel}. Chen, Zhao and Zhao established the analogous rigidity for Witt algebras and higher-rank centerless generalized Virasoro algebras \cite{ChenZhaoZhao}. Closely related two-local results are known for Schr\"odinger-type Lie algebras: Wang and Tang treated the $(2+1)$-dimensional Schr\"odinger algebra \cite{WangTang}, while Yao and Yu proved that every $2$-local derivation of the $n$-th Schr\"odinger algebra is a derivation \cite{YaoYu}. These results provide useful comparison points for the conformal Galilei family, where semidirect-product structure and central cocycles interact simultaneously.

A number of related works of Yusupov and coauthors further illustrate this dichotomy. Adashev and Yusupov first described local derivations of naturally graded quasi-filiform Leibniz algebras \cite{AdYQF} and later showed that direct sums of null-filiform Leibniz algebras generally admit local derivations which are not derivations \cite{AdYNull}. In the associative setting, Abdurasulov, Ayupov and Yusupov described local and $2$-local derivations on null-filiform, filiform and naturally graded quasi-filiform associative algebras and obtained further families with proper local derivations \cite{AAYAssoc}. The local-type program has also been developed for perfect Lie algebras, where Alauadinov, Normatov and Yusupov studied $2$-local derivations \cite{ANYLm}, and for Lie superalgebras, where Reymbaeva, Vaisova and Yusupov obtained local super-derivation results for the $n$-th super Schr\"odinger algebras \cite{RVY}. Related generalized local mappings include local and $2$-local anti-derivations on solvable Lie algebras studied by Ayupov, Atajonov and Yusupov \cite{AAtYAnti}. More recently, Ayupov and Yusupov studied local mappings on naturally graded quasi-filiform Leibniz algebras, proving in particular that the space of local derivations is closed under the commutator bracket and hence forms a finite-dimensional Lie algebra \cite{AYQF}; complementary local-automorphism phenomena for $p$-filiform Leibniz algebras were investigated in \cite{YAutp}. This structural point is especially relevant for the final part of the present paper, where the Lie algebra structure of $\LocDer$ is determined explicitly.

The Lie algebra property of local derivations is part of a broader structural program. Hao and Chen recently introduced pointwise invariant sets as a common framework for local and almost inner derivations and established general results concerning the structure of $\LocDer$ and its relation to derivation spaces \cite{HC}. Their viewpoint emphasizes that local derivations should be studied not only as individual operators but also through the algebraic structure of the whole local-derivation space. In the present conformal Galilei setting, this viewpoint is realized concretely: even in the exceptional non-central case, the additional pure-local component turns out to be an abelian ideal.

The conformal Galilei algebras provide a particularly useful family in which to compare these phenomena. They are non-semisimple Lie algebras depending on a positive half-integer $\ell$ and the spatial dimension $d$. Algebraically, the non-centrally extended algebra is
\[
\g_\ell(d)=\bigl(\mathfrak{sl}_2(\C)\oplus\mathfrak{so}_d(\C)\bigr)\ltimes V_{\ell,d},
\]
where $V_{\ell,d}$ is an abelian module obtained from the $(2\ell+1)$-dimensional irreducible $\mathfrak{sl}_2$-module together with the natural spatial module. Besides the algebra without central extension, two distinguished one-dimensional central extensions arise naturally: the mass extension for $\ell\in\mathbb N-\frac12$ and the exotic extension for integer $\ell$ in spatial dimension $d=2$. Their radicals become Heisenberg-type ideals, and the central brackets pair complementary $\mathfrak{sl}_2$-levels.

The algebraic structure of conformal Galilei symmetries has been studied from several complementary viewpoints. Andrzejewski, Gonera and Ma\'{s}lanka constructed dynamical realizations of nonrelativistic conformal groups \cite{Andrzejewski}, and Galajinsky and Masterov developed a dynamical realization of the $\ell$-conformal Galilei algebra in terms of oscillator systems \cite{GalajinskyMasterov}. For the central extensions, Aizawa and Isaac studied irreducible representations of the exotic conformal Galilei algebra \cite{AizawaIsaac}, while Aizawa, Isaac and Kimura analyzed highest-weight representations and Kac determinants for conformal Galilei algebras with central extension \cite{AizawaIsaacKimura}. More recently, Chen, Li and Yu determined derivations and $\frac12$-derivations of conformal Galilei algebras and studied the associated transposed Poisson structures \cite{ChenLiYu}. These works underline that the central extensions are structurally intrinsic objects rather than auxiliary modifications, and they motivate examining whether the same cocycles affect local derivation rigidity.

Local derivations on the closely related Schr\"odinger Lie algebras have been investigated in several papers by Alauadinov and Yusupov. The low-dimensional cases were treated in \cite{AYSch}, while the higher-dimensional Schr\"odinger algebra in $(n+1)$-dimensional space-time was studied in \cite{AYSchn}, where the local-to-global rigidity was obtained uniformly for $n\ge3$. For the conformal Galilei family itself, Alauadinov and Yusupov studied the one-dimensional mass central extension in \cite{AYCG}, proving rigidity away from the two low parameters $\ell=\frac12$ and $\ell=\frac32$. The corresponding $2$-local problem for this one-dimensional conformal Galilei algebra was considered by Zhao and Cheng \cite{ZhaoChengCG}. In the present paper the two missing local-derivation cases are computed directly, so the one-string mass result becomes complete for every admissible half-integer $\ell$. This completed picture suggests that the Heisenberg central pairing has a strong rigidity effect and motivates a simultaneous comparison with the algebra before central extension and with the exotic central extension.

The purpose of the present paper is therefore to study, within one framework, the local derivations of
\begin{enumerate}[label=\textup{(\arabic*)}]
    \item the conformal Galilei algebra $\g_\ell(d)$ without central extension;
    \item the mass central extension $\widehat{\g}^{\,M}_\ell(d)$;
    \item the exotic central extension $\widehat{\g}^{\,\Theta}_\ell(2)$.
\end{enumerate}
The comparison shows that the central cocycle changes the local-derivation problem in an essential way. In the non-central algebra the radical is abelian, and in one spatial dimension a genuine pure-local component may survive. For the mass and exotic extensions, the central pairing between complementary levels yields additional compatibility relations and restores rigidity.

Our method combines three ingredients. First, we determine the derivation algebras of the three families and isolate their outer derivations, giving an explicit pointwise form for local derivations. Second, in the central extensions we normalize on the $\mathfrak{sl}_2$-part and use carefully chosen sums of complementary weight vectors. In the mass case, paired spatial vectors force the remaining same-level coefficients to come from a single rotation operator; in the exotic case, the two weight strings and the central cocycle produce the necessary recurrence relations. Third, in the non-central case the remaining problem is reduced to algebraic reflexivity and local coboundaries, which isolates the exceptional pure-local module.

The principal result is the following classification.

\begin{theorem}\label{thm:main-intro}
Let $\ell\in\frac12\mathbb N$ and $d\in\mathbb N$.
\begin{enumerate}[label=\textup{(\roman*)}]
\item For the algebra without central extension, every local derivation is a derivation if $d\ge2$. If $d=1$, the same is true when $2\ell$ is even or $\ell=\frac12$. For odd $2\ell\ge3$,
\[
\LocDer\g_\ell(1)=\Der\g_\ell(1)\oplus\mathcal P_{2\ell},
\qquad \dim\mathcal P_{2\ell}=2\ell-1.
\]
Moreover, $\mathcal P_{2\ell}$ is an abelian irreducible $\mathfrak{sl}_2$-module, isomorphic to the module of highest weight $2\ell-2$.
\item If $\ell\in\mathbb N-\frac12$, then every local derivation on the mass central extension $\widehat{\g}^{\,M}_\ell(d)$ is a derivation for every $d$.
\item If $d=2$ and $\ell\in\mathbb N$, then every local derivation on the exotic central extension $\widehat{\g}^{\,\Theta}_\ell(2)$ is a derivation.
\end{enumerate}
\end{theorem}

Consequently, the two central extensions are locally rigid, whereas the non-central algebra possesses a precisely determined exceptional family in spatial dimension one. For odd $2\ell\ge3$, passing from $\g_\ell(1)$ to its mass central extension eliminates a $(2\ell-1)$-dimensional space of pure local derivations. We also determine the Lie algebra structure of the full local-derivation space; in the exceptional case the pure-local component is an abelian ideal. Thus the paper complements the previously known rigid examples by giving, inside a single natural family, both a complete rigidity theorem and an explicit representation-theoretic description of the obstruction to rigidity.

The paper is organized as follows. Section~\ref{sec:three} introduces the three conformal Galilei Lie algebras and fixes notation. Section~\ref{sec:derivations} determines their derivation algebras. Section~\ref{sec:local} contains the classification of local derivations, including the algebraic-reflexivity argument for the non-central case and the rigidity proofs for the mass and exotic extensions. The final structural section describes dimensions, Lie brackets, and the Levi decomposition of the local-derivation algebra in the exceptional case.

\section{The three conformal Galilei Lie algebras}\label{sec:three}
In this section we fix the notation and recall the defining relations of the three Lie algebras considered throughout the paper. We begin with the algebra without a central extension and then introduce the mass and exotic one-dimensional central extensions, emphasizing the structural features that will later control their derivations and local derivations.

\subsection{The algebra without central extension}
We first describe the basic, unextended conformal Galilei algebra. Its semidirect-product structure provides the common framework from which both central extensions will be obtained and also makes the role of the abelian radical transparent.
Let $d\in\mathbb N$ and $\ell\in\frac12\mathbb N$. The conformal Galilei algebra $\g_\ell(d)$ over $\C$ has basis
\[
\{D,H,C,M_{ij}=-M_{ji},P_i^{(n)}\mid 1\le i,j\le d,\ 0\le n\le 2\ell\}.
\]
The multiplication table is given by the following non-zero brackets:
\begin{equation}\label{eq:noncentral-table}
\left\{
\begin{aligned}
[D,H]&=2H, \qquad [D,C]=-2C, \qquad [C,H]=D,\\[1mm]
[M_{ij},M_{kl}]&=-\delta_{ik}M_{jl}-\delta_{jl}M_{ik}
+\delta_{il}M_{jk}+\delta_{jk}M_{il},\\[1mm]
[H,P_i^{(n)}]&=-nP_i^{(n-1)}, \qquad
[D,P_i^{(n)}]=2(\ell-n)P_i^{(n)},\\[1mm]
[C,P_i^{(n)}]&=(2\ell-n)P_i^{(n+1)},\\[1mm]
[M_{ij},P_k^{(n)}]&=-\delta_{ik}P_j^{(n)}+\delta_{jk}P_i^{(n)}.
\end{aligned}
\right.
\end{equation}
Here $1\le i,j,k,l\le d$, $0\le n\le 2\ell$, and $\delta_{ij}$ denotes the Kronecker delta. As usual,
\[
P_i^{(-1)}=P_i^{(2\ell+1)}=0,
\]
and all brackets not obtained from \eqref{eq:noncentral-table} by skew-symmetry are zero. Put
\[
V_{\ell,d}=\operatorname{span}_{\C}\{P_i^{(n)}\}.
\]
Then $V_{\ell,d}$ is an abelian ideal and
\[
\g_\ell(d)=\bigl(\slc\oplus\sod\bigr)\ltimes V_{\ell,d},
\]
with the standard interpretation for $d=1$ and $d=2$.

\subsection{Mass central extension}
We next introduce the mass central extension, which replaces the abelian radical by a Heisenberg-type ideal through a nontrivial central pairing of complementary levels. This modification will be responsible for a substantial increase in rigidity in the local-derivation problem.

Assume $\ell\in\mathbb N-\frac12$. The mass extension
$\widehat{\g}^{\,M}_\ell(d)=\g_\ell(d)\oplus \C M$
has the additional central bracket
\begin{equation}\label{eq:mass}
[P_i^{(m)},P_j^{(n)}]
=\delta_{ij}\delta_{m+n,2\ell}I_m M,
\qquad
I_m=(-1)^{m+\ell+\frac12}(2\ell-m)!\,m!,
\end{equation}
where $M$ is central.

\subsection{Exotic central extension}
The second central extension considered here is the exotic extension, which exists in two spatial dimensions for integral $\ell$. Its alternating spatial cocycle is different from the mass cocycle, but it again couples complementary weight spaces through a central element.

Assume $d=2$ and $\ell\in\mathbb N$. The exotic extension
$\widehat{\g}^{\,\Theta}_\ell(2)=\g_\ell(2)\oplus\C\Theta$
has the additional bracket
\begin{equation}\label{eq:exotic}
[P_i^{(m)},P_j^{(n)}]
=\varepsilon_{ij}\delta_{m+n,2\ell}J_m\Theta,
\qquad
J_m=(-1)^m(2\ell-m)!\,m!,
\end{equation}
where $\Theta$ is central and $\varepsilon_{12}=1$.

For the direct calculations below, introduce the complex basis
\[
h=D,\qquad e=-H,\qquad f=C,\qquad s=-\mathrm i M_{12},
\]
\[
p_n=P_1^{(n)}+\mathrm iP_2^{(n)},\qquad
q_n=P_1^{(n)}-\mathrm iP_2^{(n)},\qquad
z=-2\mathrm i\Theta.
\]
Then
\begin{align}
[h,e]&=2e,& [h,f]&=-2f,& [e,f]&=h,\label{eq:ex1}\\
[h,p_n]&=2(\ell-n)p_n,& [e,p_n]&=np_{n-1},& [f,p_n]&=(2\ell-n)p_{n+1},\\
[h,q_n]&=2(\ell-n)q_n,& [e,q_n]&=nq_{n-1},& [f,q_n]&=(2\ell-n)q_{n+1},\\
[s,p_n]&=p_n,& [s,q_n]&=-q_n,&
[p_m,q_n]&=\delta_{m+n,2\ell}(-1)^m(2\ell-m)!m!\,z.\label{eq:ex4}
\end{align}
Thus
\[
\widehat{\g}^{\,\Theta}_\ell(2)
=(\slc\oplus\mathfrak{so}_2)\ltimes \mathfrak h_{\ell},
\]
where $\mathfrak h_\ell$ is the Heisenberg ideal generated by $p_n,q_n,z$.

\section{Derivations of the three conformal Galilei Lie algebras}\label{sec:derivations}
In this section we determine the derivation algebras of all three Lie algebras introduced above.  The results are arranged as separate subsections so that the effect of the central cocycle is visible at once.

\subsection{The algebra without central extension}\label{subsec:der-noncentral}
We first determine the ordinary derivations of the algebra with abelian radical. The calculation separates the inner derivations from a single natural outer scaling on the radical, and this decomposition will later serve as the starting point for the local analysis.

Define the linear map $\delta_0:\g_\ell(d)\to\g_\ell(d)$ by
\begin{equation}\label{eq:delta0}
\delta_0(D)=\delta_0(H)=\delta_0(C)=\delta_0(M_{ij})=0,
\qquad
\delta_0(P_i^{(n)})=P_i^{(n)}.
\end{equation}

\begin{lemma}\label{lem:delta0}
The map $\delta_0$ is a derivation of $\g_\ell(d)$.
\end{lemma}
\begin{proof}
The only point requiring verification is compatibility with the semidirect-product action. If $x\in\slc\oplus\sod$ and $v\in V_{\ell,d}$, then
\[
\delta_0([x,v])=[x,v]=[\delta_0(x),v]+[x,\delta_0(v)].
\]
Since $V_{\ell,d}$ is abelian, the derivation identity is automatic on pairs from $V_{\ell,d}$, while it is trivial on the reductive part because $\delta_0$ vanishes there.
\end{proof}

\begin{theorem}\label{thm:der0}
For the conformal Galilei algebra without central extension,
\begin{equation}\label{eq:der0}
\Der\g_\ell(d)=\Inn\g_\ell(d)\oplus\C\delta_0.
\end{equation}
\end{theorem}
\begin{proof}
Put $L=\g_\ell(d)$ and $V=V_{\ell,d}$. The ideal $V$ is the nilradical of $L$, hence it is characteristic. Therefore every derivation preserves $V$ and induces a derivation of $L/V$.

For $d\neq2$, the quotient is semisimple (with the usual $d=1$ interpretation). After subtracting an inner derivation, we may assume that the derivation vanishes on the Levi factor. The remaining restriction to $V$ commutes with the $\slc\oplus\sod$-action. The module $V$ is irreducible, and therefore Schur's lemma yields a common scalar action on $V$. This scalar action is precisely a multiple of $\delta_0$.

For $d=2$, put $s=-\mathrm iM_{12}$ and write
\[
V=V^+\oplus V^-,\qquad
V^+=\operatorname{span}\{p_0,\ldots,p_{2\ell}\},\qquad
V^-=\operatorname{span}\{q_0,\ldots,q_{2\ell}\}.
\]
After subtracting an inner derivation, we may assume first that the derivation vanishes on $\mathfrak{sl}_2$. Since $V$ has no non-zero $\mathfrak{sl}_2$-invariant vector, the relation $[\mathfrak{sl}_2,s]=0$ implies
\[
D(s)=a s.
\]
The restriction of $D$ to $V$ is $\mathfrak{sl}_2$-equivariant, and therefore
\[
D(p_n)=\alpha p_n+\beta q_n,\qquad
D(q_n)=\gamma p_n+\delta q_n.
\]
Applying $D$ to $[s,p_n]=p_n$ gives
\[
\alpha p_n+\beta q_n=(a+\alpha)p_n-\beta q_n,
\]
so $a=\beta=0$. Similarly, $[s,q_n]=-q_n$ gives $\gamma=0$. Thus the two strings are scaled independently. Their common scalar part is a multiple of $\delta_0$, while their difference is a multiple of the inner derivation $\ad s$, because
\[
\ad s(p_n)=p_n,\qquad \ad s(q_n)=-q_n.
\]
Hence, modulo inner derivations, only the common scalar survives, proving \eqref{eq:der0} also for $d=2$.
\end{proof}

\subsection{The mass central extension}\label{subsec:der-mass}
We now pass to the mass extension and determine how the central cocycle changes the derivation algebra. Besides modifying the weight of the scalar outer derivation on the center, the two-dimensional case produces an additional central shear.

Define $\delta_M:\widehat{\g}^{\,M}_\ell(d)\to\widehat{\g}^{\,M}_\ell(d)$ by
\begin{equation}\label{eq:deltaM}
\delta_M(P_i^{(n)})=P_i^{(n)},\qquad
\delta_M(M)=2M,
\end{equation}
and let $\delta_M$ vanish on $D,H,C,M_{ij}$.

\begin{lemma}\label{lem:deltaM}
The map $\delta_M$ is a derivation of $\widehat{\g}^{\,M}_\ell(d)$.
\end{lemma}
\begin{proof}
For the non-central brackets the verification is the same as in Lemma~\ref{lem:delta0}. For the cocycle \eqref{eq:mass},
\[
\delta_M([P_i^{(m)},P_j^{(n)}])=2[P_i^{(m)},P_j^{(n)}],
\]
whereas
\[
[\delta_M(P_i^{(m)}),P_j^{(n)}]+[P_i^{(m)},\delta_M(P_j^{(n)})]
=2[P_i^{(m)},P_j^{(n)}].
\]
Thus the coefficient $2$ on the central element is forced by the derivation identity.
\end{proof}

When $d=2$, put $s=-\mathrm iM_{12}$ and define
\begin{equation}\label{eq:etaM}
\eta_M(s)=M,
\qquad
\eta_M(x)=0
\quad\text{for every other basis element }x.
\end{equation}
Since $M$ is central, $\eta_M$ is a derivation.

\begin{theorem}\label{thm:massder}
For $d\neq2$,
\begin{equation}\label{eq:massder1}
\Der\widehat{\g}^{\,M}_\ell(d)
=\Inn\widehat{\g}^{\,M}_\ell(d)\oplus\C\delta_M.
\end{equation}
For $d=2$,
\begin{equation}\label{eq:massder2}
\Der\widehat{\g}^{\,M}_\ell(2)
=\Inn\widehat{\g}^{\,M}_\ell(2)
\oplus\C\delta_M\oplus\C\eta_M.
\end{equation}
\end{theorem}
\begin{proof}
Let $L=\widehat{\g}^{\,M}_\ell(d)$ and
\[
R=V_{\ell,d}\oplus\C M.
\]
Then $R$ is the nilradical of $L$, hence $R$ is characteristic. Moreover $Z(L)=\C M$.

First assume $d\neq2$. The quotient $L/R$ is semisimple. By subtracting an inner derivation and using the usual Whitehead-lemma reduction, we may suppose that a given $D\in\Der L$ vanishes on the Levi factor. Then the induced map on $R/\C M\cong V_{\ell,d}$ commutes with the Levi action. Since $V_{\ell,d}$ is irreducible, Schur's lemma gives
\[
D(P_i^{(n)})=aP_i^{(n)}+\phi_i^{(n)}M.
\]
The central coefficients define an invariant linear functional on the non-trivial irreducible module $V_{\ell,d}$; hence all $\phi_i^{(n)}$ vanish. Since the center is characteristic, $D(M)=\rho M$. Applying $D$ to a non-zero bracket
\(
[P_i^{(m)},P_i^{(2\ell-m)}]
\)
gives $\rho=2a$. Therefore the normalized derivation equals $a\delta_M$.

Now let $d=2$. Put
\[
s=-\mathrm iM_{12},\qquad
p_n=P_1^{(n)}+\mathrm iP_2^{(n)},\qquad
q_n=P_1^{(n)}-\mathrm iP_2^{(n)}.
\]
Then
\[
[s,p_n]=p_n,\qquad [s,q_n]=-q_n,
\]
and the mass cocycle becomes
\[
[p_r,p_t]=[q_r,q_t]=0,\qquad
[p_r,q_t]=2\delta_{r+t,2\ell}I_r M.
\]
Use the $s$-weight decomposition
\[
V=V^+\oplus V^-,\qquad
V^+=\operatorname{span}\{p_n\},\qquad
V^-=\operatorname{span}\{q_n\}.
\]
After subtracting an inner derivation we may assume that $D(e)=D(f)=D(h)=0$.
Since $[s,\mathfrak{sl}_2]=0$, applying $D$ to these relations shows that $D(s)$ centralizes
$\mathfrak{sl}_2$.  The module $V$ contains no trivial $\mathfrak{sl}_2$-submodule, so
\[
D(s)=as+bM.
\]
The identity obtained from $[s,v]$ on $V/\C M$ is
\[
T\,\ad(s)-\ad(s)\,T=a\,\ad(s),
\]
where $T$ is the induced endomorphism of $V$. Comparing the diagonal blocks on the $\pm1$ eigenspaces of $\ad(s)$ gives $a=0$. Hence
\[
D(s)=bM.
\]

The restriction of $D$ to $V$ is $\mathfrak{sl}_2$-equivariant. Hence, before using the $s$-action, it has the form
\[
D(p_n)=\alpha p_n+\beta q_n,\qquad
D(q_n)=\gamma p_n+\delta q_n.
\]
Applying $D$ to $[s,p_n]=p_n$ and $[s,q_n]=-q_n$ yields $\beta=\gamma=0$. No additional central terms occur because they would define $\mathfrak{sl}_2$-invariant functionals on the non-trivial irreducible modules $V^+$ and $V^-$. Thus
\[
D(p_n)=\alpha p_n,\qquad D(q_n)=\delta q_n.
\]
The mass cocycle gives $D(M)=(\alpha+\delta)M$. Taking
\[
\lambda=\frac{\alpha+\delta}{2},\qquad
\mu=\frac{\alpha-\delta}{2},
\]
and using
\[
\ad s(p_n)=p_n,\qquad \ad s(q_n)=-q_n,
\]
we obtain
\[
D=b\eta_M+\lambda\delta_M+\mu\ad s
\]
after normalization. Restoring the inner derivations proves \eqref{eq:massder2}. The outer maps $\delta_M$ and $\eta_M$ are plainly independent modulo inner derivations, so the sum is direct.
\end{proof}

\subsection{The exotic central extension}\label{subsec:der-exotic}
We finally treat the exotic extension. The complex basis adapted to the rotation generator and to the two conjugate radical strings makes the cocycle relations particularly efficient for determining all outer derivations.

We now compute the derivation algebra of $\widehat{\g}^{\,\Theta}_\ell(2)$ directly from its defining relations.

Recall the complex basis
\[
h=D,\qquad e=-H,\qquad f=C,\qquad s=-\mathrm iM_{12},
\]
\[
p_n=P_1^{(n)}+\mathrm iP_2^{(n)},\qquad
q_n=P_1^{(n)}-\mathrm iP_2^{(n)},\qquad
z=-2\mathrm i\Theta.
\]
The non-zero brackets are
\begin{equation}\label{eq:exotic-table}
\left\{
\begin{aligned}
[h,e]&=2e, & [h,f]&=-2f, & [e,f]&=h,\\[1mm]
[h,p_n]&=2(\ell-n)p_n, & [e,p_n]&=np_{n-1}, & [f,p_n]&=(2\ell-n)p_{n+1},\\[1mm]
[h,q_n]&=2(\ell-n)q_n, & [e,q_n]&=nq_{n-1}, & [f,q_n]&=(2\ell-n)q_{n+1},\\[1mm]
[s,p_n]&=p_n, & [s,q_n]&=-q_n,\\[1mm]
[p_m,q_n]&=\delta_{m+n,2\ell}(-1)^m(2\ell-m)!m!\,z.
\end{aligned}
\right.
\end{equation}
Here $0\le m,n\le2\ell$ and $p_{-1}=q_{-1}=p_{2\ell+1}=q_{2\ell+1}=0$.

Define linear maps $\eta_\Theta$ and $\delta_\Theta$ by
\begin{equation}\label{eq:etaTheta}
\eta_\Theta(s)=z,
\qquad
\eta_\Theta(x)=0\quad(x\neq s),
\end{equation}
and
\begin{equation}\label{eq:deltaTheta}
\delta_\Theta(p_n)=p_n,\qquad
\delta_\Theta(q_n)=q_n,\qquad
\delta_\Theta(z)=2z,
\end{equation}
with $\delta_\Theta(e)=\delta_\Theta(f)=\delta_\Theta(h)=\delta_\Theta(s)=0$.

\begin{lemma}\label{lem:exotic-outer}
The maps $\eta_\Theta$ and $\delta_\Theta$ are derivations of $\widehat{\g}^{\,\Theta}_\ell(2)$.
\end{lemma}
\begin{proof}
Since $z$ is central and $s$ does not occur on the right-hand side of any non-zero bracket, the derivation identity for $\eta_\Theta$ follows immediately.

For $\delta_\Theta$, all brackets involving $e,f,h,s$ and one of $p_n,q_n$ are preserved because $\delta_\Theta$ acts as the identity on both strings. For the Heisenberg bracket,
\[
\delta_\Theta([p_m,q_n])=2[p_m,q_n]
=[\delta_\Theta(p_m),q_n]+[p_m,\delta_\Theta(q_n)].
\]
Thus $\delta_\Theta$ is a derivation.
\end{proof}

\begin{theorem}\label{thm:exoticder}
For every positive integer $\ell$,
\begin{equation}\label{eq:exoticder}
\Der\widehat{\g}^{\,\Theta}_\ell(2)
=
\Inn\widehat{\g}^{\,\Theta}_\ell(2)
\oplus\C\eta_\Theta
\oplus\C\delta_\Theta.
\end{equation}
\end{theorem}
\begin{proof}
Let $L=\widehat{\g}^{\,\Theta}_\ell(2)$ and let $D\in\Der L$. The subalgebra
\[
\mathfrak s=\operatorname{span}\{e,f,h\}\cong\mathfrak{sl}_2(\C)
\]
is semisimple. By Whitehead's lemma there exists $u\in L$ such that, after replacing $D$ by $D-\ad u$,
\begin{equation}\label{eq:ex-normal1}
D(e)=D(f)=D(h)=0.
\end{equation}
Because $[\mathfrak s,s]=0$, relation \eqref{eq:ex-normal1} gives $[x,D(s)]=0$ for $x\in\mathfrak s$. Hence
\[
D(s)=as+bz.
\]

The restriction of $D$ to
\[
V^+=\operatorname{span}\{p_0,\ldots,p_{2\ell}\},\qquad
V^-=\operatorname{span}\{q_0,\ldots,q_{2\ell}\}
\]
is $\mathfrak{sl}_2$-equivariant.  No additional $z$-component can occur, because it would define an $\mathfrak{sl}_2$-homomorphism from the non-trivial irreducible module $U_{2\ell}$ to the trivial module $\C z$.  Hence
\[
D(p_n)=\alpha p_n+\beta q_n,\qquad
D(q_n)=\gamma p_n+\delta q_n.
\]
Applying $D$ to $[s,p_n]=p_n$ gives
\[
\alpha p_n+\beta q_n
=(a+\alpha)p_n-\beta q_n,
\]
so $a=\beta=0$. Similarly, the relation $[s,q_n]=-q_n$ yields $\gamma=0$. Therefore
\begin{equation}\label{eq:ex-normal2}
D(s)=bz,\qquad
D(p_n)=\alpha p_n,\qquad
D(q_n)=\delta q_n.
\end{equation}
Since $Z(L)=\C z$, write $D(z)=\rho z$. Applying $D$ to a non-zero bracket $[p_m,q_{2\ell-m}]$ gives
\[
\rho=\alpha+\delta.
\]
Set
\[
\lambda=\frac{\alpha+\delta}{2},\qquad
\mu=\frac{\alpha-\delta}{2}.
\]
Using
\[
\ad s(p_n)=p_n,\qquad \ad s(q_n)=-q_n,\qquad \ad s(z)=0,
\]
we obtain
\[
D=b\eta_\Theta+\lambda\delta_\Theta+\mu\ad s
\]
after the normalization \eqref{eq:ex-normal1}. Restoring the inner derivation subtracted at the beginning proves \eqref{eq:exoticder}. Finally, $\eta_\Theta$ and $\delta_\Theta$ are linearly independent modulo $\Inn L$, so the sum is direct.
\end{proof}

\subsection{Comparison of the three derivation algebras}\label{subsec:der-comparison}
We conclude the derivation section by placing the three descriptions side by side. This comparison isolates precisely which outer directions survive the introduction of a central cocycle and prepares the normalization used for local derivations.

The preceding theorems show exactly how the central cocycle changes the outer derivation space.  In the non-central case there is one scalar outer derivation.  The mass cocycle forces the center to have twice the scalar weight, and in dimension $2$ it also permits the central shear $s\mapsto M$.  The exotic cocycle has the same scalar-weight phenomenon and the analogous shear $s\mapsto z$.  Thus
\[
\begin{aligned}
\Der\g_\ell(d)/\Inn\g_\ell(d)&\cong\C,\\
\Der\widehat{\g}^{\,M}_\ell(d)/\Inn\widehat{\g}^{\,M}_\ell(d)&\cong\C \qquad (d\neq2),\\
\Der\widehat{\g}^{\,M}_\ell(2)/\Inn\widehat{\g}^{\,M}_\ell(2)&\cong\C^2,\\
\Der\widehat{\g}^{\,\Theta}_\ell(2)/\Inn\widehat{\g}^{\,\Theta}_\ell(2)&\cong\C^2.
\end{aligned}
\]

\section{Local derivations}\label{sec:local}
Having determined the full derivation algebras, we now turn to the local-to-global problem. The aim of this section is to decide, in each of the three settings, whether pointwise agreement with ordinary derivations forces a linear map to be a derivation and, when it does not, to describe the additional pure local component explicitly.

\begin{definition}
A linear map $\Delta:L\to L$ is called a local derivation if for every $x\in L$ there exists $D_x\in\Der L$ such that
\[
\Delta(x)=D_x(x).
\]
\end{definition}

The descriptions in Section~\ref{sec:derivations} give a pointwise normal form for every local derivation.  We now record the first coefficient restrictions that will be used in the complete proof.

\subsection{Local derivations on \texorpdfstring{$\g_\ell(d)$}{g-l(d)}}\label{subsec:local-noncentral}
We begin with the algebra without central extension, where the radical is abelian and therefore allows the greatest possible local freedom. The problem can be reformulated in terms of range-local coboundaries for the $\mathfrak{sl}_2\oplus\mathfrak{so}_d$-module carried by the radical.

Put $m=2\ell$ and write
\[
U_m=\operatorname{span}_{\C}\{u_0,\ldots,u_m\},\qquad
W_d=\C^d,
\]
where $u_k$ denotes the $\mathfrak{sl}_2$-weight vector of weight $m-2k$.  Thus
\[
V_{\ell,d}\cong U_m\otimes W_d,
\]
and the $\mathfrak{sl}_2$-action on $U_m$ is
\begin{equation}\label{eq:Um-action}
 h u_k=(m-2k)u_k,\qquad
 e u_k=k u_{k-1},\qquad
 f u_k=(m-k)u_{k+1},
\end{equation}
with $u_{-1}=u_{m+1}=0$.  Here, as before,
$h=D$, $e=-H$, and $f=C$.

By Theorem~\ref{thm:der0}, for every $x\in\g_\ell(d)$ there exist
$a_x\in\g_\ell(d)$ and $\lambda_x\in\C$ such that
\begin{equation}\label{eq:local0}
\Delta(x)=[a_x,x]+\lambda_x\delta_0(x).
\end{equation}

\begin{lemma}\label{lem:local0-support}
Let $\Delta\in\LocDer\g_\ell(d)$. Then
\begin{align*}
\Delta(H)&\in\operatorname{span}\{D,H,P_i^{(r)}:1\le i\le d,\ 0\le r\le2\ell-1\},\\
\Delta(C)&\in\operatorname{span}\{D,C,P_i^{(r)}:1\le i\le d,\ 1\le r\le2\ell\},\\
\Delta(D)&\in\operatorname{span}\{H,C,P_i^{(r)}:1\le i\le d,\ r\neq\ell\},
\end{align*}
and, for $0\le k\le2\ell$,
\begin{equation}\label{eq:local0-Psupport}
\Delta(P_i^{(k)})\in
\operatorname{span}\Bigl(
\{P_i^{(k-1)},P_i^{(k+1)}\}\cup\{P_j^{(k)}:1\le j\le d\}
\Bigr),
\end{equation}
where the terms with indices $-1$ and $2\ell+1$ are omitted.
\end{lemma}
\begin{proof}
This follows immediately from \eqref{eq:local0} and the multiplication table
\eqref{eq:noncentral-table}.  The outer derivation $\delta_0$ vanishes on the reductive part and acts as the identity on $V_{\ell,d}$.  The $e$- and $f$-components shift the level by one, $h$ preserves the level, and the rotation generators change only the spatial index.  Since $V_{\ell,d}$ is abelian, no central component can occur.
\end{proof}

\begin{lemma}\label{lem:local-on-S}
Let
\[
S=\mathfrak{sl}_2\oplus\mathfrak{so}_d.
\]
Every local derivation on $S$ is a derivation.
\end{lemma}
\begin{proof}
If $d\neq2$, then $S$ is semisimple (with $\mathfrak{so}_1=0$), and the assertion is the standard finite-dimensional semisimple result for local derivations.

Let $d=2$ and write $S=\mathfrak{sl}_2\oplus\C s$.  Every derivation of $S$ preserves both the derived ideal
$[S,S]=\mathfrak{sl}_2$ and the center $Z(S)=\C s$.  Hence, if $\Lambda$ is local, then
\[
\Lambda(\mathfrak{sl}_2)\subseteq\mathfrak{sl}_2,
\qquad
\Lambda(s)\in\C s.
\]
The restriction $\Lambda|_{\mathfrak{sl}_2}$ is a local derivation of $\mathfrak{sl}_2$, so
$\Lambda|_{\mathfrak{sl}_2}=\ad a$ for some $a\in\mathfrak{sl}_2$.  Writing
$\Lambda(s)=\mu s$, we obtain
\[
\Lambda=\ad a+\mu\vartheta,
\qquad
\vartheta|_{\mathfrak{sl}_2}=0,
\quad
\vartheta(s)=s,
\]
which is a derivation of $S$.
\end{proof}

The next two lemmas settle the operator-space reflexivity needed for the module part of a local derivation.

\begin{lemma}\label{lem:sl2-reflexive}
Let
\[
\mathcal A_m=\rho_m(\mathfrak{sl}_2)+\C I_{U_m}\subseteq\operatorname{End}(U_m).
\]
If $T\in\operatorname{End}(U_m)$ satisfies
\[
T(u)\in\mathcal A_m u\qquad(u\in U_m),
\]
then $T\in\mathcal A_m$.
\end{lemma}
\begin{proof}
For $m=1$, one has $\mathcal A_1=\operatorname{End}(U_1)$, so there is nothing to prove. Assume $m\ge2$ and consider the rational normal curve
\[
v(t)=\sum_{k=0}^{m}\binom{m}{k}t^k u_k.
\]
The space $\mathcal A_m v(t)$ is the tangent space to the affine cone over this curve, hence
\[
T v(t)=\alpha(t)v(t)+\beta(t)v'(t)
\]
for suitable scalars depending on $t$.  Let $w_k(t)$ be the $u_k$-coordinate of $Tv(t)$. Since
$w_0=\alpha$ and $w_1=m(\alpha t+\beta)$, for $1\le k\le m$ we obtain
\begin{equation}\label{eq:rnc-identity}
\frac{w_k(t)}{\binom{m}{k}}
=\frac{k}{m}t^{k-1}w_1(t)-(k-1)t^k w_0(t).
\end{equation}
Each $w_k$ has degree at most $m$.  Taking $k=m$ in \eqref{eq:rnc-identity}, and for $m\ge3$ also $k=m-1$, comparison of the coefficients of degrees greater than $m$ yields
\[
w_0(t)=a_0+a_1t,
\qquad
w_1(t)=b_0+b_1t+(m-1)a_1t^2.
\]
(The case $m=2$ follows already from the equation with $k=m$.) Substitution into
\eqref{eq:rnc-identity} gives
\[
\frac{w_k(t)}{\binom{m}{k}}
=\frac{k b_0}{m}t^{k-1}
+\left(a_0+k\left(\frac{b_1}{m}-a_0\right)\right)t^k
+\frac{(m-k)a_1}{m}t^{k+1}.
\]
Comparing with
$w_k(t)=\sum_j T_{kj}\binom{m}{j}t^j$ shows that
\[
Tu_j=\frac{a_1}{m}j u_{j-1}
+\left(\alpha+\beta(m-2j)\right)u_j
+\frac{b_0}{m}(m-j)u_{j+1}
\]
for suitable $\alpha,\beta\in\C$. Hence
$T\in\operatorname{span}\{I,e,f,h\}=\mathcal A_m$.
\end{proof}

\begin{lemma}\label{lem:so-reflexive}
Let
\[
\mathcal B_d=\mathfrak{so}_d(\C)+\C I_{W_d}\subseteq\operatorname{End}(W_d).
\]
If $Cw\in\mathcal B_d w$ for every $w\in W_d$, then $C\in\mathcal B_d$.
\end{lemma}
\begin{proof}
The statement is trivial for $d=1$. For $d=2$, the two isotropic lines for the quadratic form
$q(w)=w_1^2+w_2^2$ must be invariant under $C$; hence $C$ is a linear combination of the identity and the standard skew generator.

Assume $d\ge3$.  If $q(w)\ne0$, then
$\C w+\mathfrak{so}_d w=W_d$.  For a non-zero isotropic vector $w$, however,
$\C w+\mathfrak{so}_d w=w^\perp$. Therefore
\[
w^{\mathsf T}Cw=0\qquad\text{whenever }q(w)=0.
\]
Only the symmetric part $S=(C+C^{\mathsf T})/2$ contributes to this quadratic polynomial. Since the irreducible quadratic form $q$ divides the quadratic form $w^{\mathsf T}Sw$, one has
$S=\lambda I$. Consequently $C-\lambda I$ is skew-symmetric, proving the claim.
\end{proof}

\begin{proposition}\label{prop:module-reflexive}
Let $S=\mathfrak{sl}_2\oplus\mathfrak{so}_d$ and let
$\rho:S\to\mathfrak{gl}(V_{\ell,d})$ be the natural action. Then
\begin{equation}\label{eq:module-reflexive}
\operatorname{Ref}\bigl(\rho(S)+\C I\bigr)=\rho(S)+\C I.
\end{equation}
\end{proposition}
\begin{proof}
Write $V_{\ell,d}=U_m\otimes W_d$ and fix a basis $\varepsilon_1,\ldots,\varepsilon_d$ of $W_d$. Let $T$ belong to the reflexive hull in \eqref{eq:module-reflexive} and write $T=(T_{ji})$ in $d\times d$ blocks, each block acting on $U_m$.

For $u\otimes\varepsilon_i$, locality gives
\[
T_{ii}u\in\mathcal A_m u,
\qquad
T_{ji}u\in\C u\quad(j\ne i).
\]
By Lemma~\ref{lem:sl2-reflexive}, $T_{ii}\in\mathcal A_m$, while the second condition implies
$T_{ji}=c_{ji}I$. Applying locality to $u\otimes(\varepsilon_i+\varepsilon_j)$ shows that
$(T_{ii}-T_{jj})u\in\C u$ for every $u$, hence $T_{ii}-T_{jj}$ is scalar. Therefore
\[
T=B\otimes I_{W_d}+I_{U_m}\otimes C
\]
for some $B\in\mathcal A_m$ and $C\in\operatorname{End}(W_d)$.

Subtract $B\otimes I$. For every non-zero $u$ and every $w$ we then have
\[
u\otimes Cw\in \mathcal A_m u\otimes w+u\otimes\mathfrak{so}_d w.
\]
Projecting to $U_m/\C u$ forces the first factor to be scalar on $u$, and hence
$Cw\in(\mathfrak{so}_d+\C I)w$. Lemma~\ref{lem:so-reflexive} now gives
$C\in\mathcal B_d$, and the result follows.
\end{proof}

\begin{lemma}\label{lem:noncentral-compatibility}
Suppose that a local derivation $\Delta$ induces the zero map on the quotient
$\g_\ell(d)/V_{\ell,d}$. If
\[
\Delta|_{V_{\ell,d}}=(\rho(b)+\lambda I)|_{V_{\ell,d}}
\]
for some $b\in S$ and $\lambda\in\C$, then $b\in Z(S)$. Thus $b=0$ for $d=1$ and for $d\ge3$, while for $d=2$ it is a scalar multiple of the generator of $\mathfrak{so}_2$.
\end{lemma}
\begin{proof}
Write $b=b_1+b_2$ with $b_1\in\mathfrak{sl}_2$ and $b_2\in\mathfrak{so}_d$. For every $x\in S$ and $v\in V_{\ell,d}$, locality at $x+t v$ and division by $t\ne0$ imply
\begin{equation}\label{eq:compatibility-inclusion}
\rho(b)v\in \rho(x)V_{\ell,d}
+\bigl(\rho(C_S(x))+\C I\bigr)v.
\end{equation}
Indeed, the fixed vector $\Delta(x)$ belongs to $\rho(x)V_{\ell,d}$ and therefore disappears modulo that image.

Take first a rotation generator $s_{ij}$ and a vector $v=u\otimes w$ with
$w\in\operatorname{span}\{\varepsilon_i,\varepsilon_j\}$.  Both
$\rho(s_{ij})V$ and $\rho(C_S(s_{ij}))v$ have their spatial part inside the same two-dimensional plane.  Hence \eqref{eq:compatibility-inclusion} forces
$b_2\operatorname{span}\{\varepsilon_i,\varepsilon_j\}\subseteq
\operatorname{span}\{\varepsilon_i,\varepsilon_j\}$ for every $i<j$.
If $d\ge3$, this implies $b_2\varepsilon_i\in\C\varepsilon_i$ for every $i$; skew-symmetry gives $b_2=0$.  If $d=2$, no further restriction occurs and $b_2\in Z(\mathfrak{so}_2)$.

It remains to treat $b_1$.  Choose a non-isotropic spatial vector $w$ and project
\eqref{eq:compatibility-inclusion} onto the $w$-component along $w^\perp$.  For a nilpotent element $x=e$ this gives
\[
\rho_m(b_1)u\in \rho_m(e)U_m+(\C\rho_m(e)+\C I)u
\qquad(u\in U_m).
\]
Taking $u=u_{m-1}$ shows that the $f$-coefficient of $b_1$ is zero. The analogous argument with $x=f$ and $u=u_1$ removes the $e$-coefficient. More invariantly, repeating the same argument for every conjugate of $e$ shows that $b_1$ belongs to every Borel subalgebra of $\mathfrak{sl}_2$. Their intersection is zero, and hence $b_1=0$.
\end{proof}

\begin{lemma}\label{lem:range-local-gl2}
Let $\Phi:\mathcal A_m\to U_m$ be linear and suppose
\[
\Phi(A)\in A U_m\qquad(A\in\mathcal A_m).
\]
Then there exists $u\in U_m$ such that $\Phi(A)=Au$ for all $A\in\mathcal A_m$.
\end{lemma}
\begin{proof}
Put $u=\Phi(I)$ and $\Psi(A)=\Phi(A)-Au$. Then $\Psi(I)=0$ and
$\Psi(A)\in AU_m$ for every $A$. Take a non-zero semisimple $x\in\mathfrak{sl}_2$.
The operator $\rho_m(x)$ is diagonalizable with the distinct eigenvalues
$m\alpha,(m-2)\alpha,\ldots,-m\alpha$ for some $\alpha\ne0$. For every $t\in\C$,
\[
\Psi(\rho_m(x))=\Psi(\rho_m(x)+tI)\in
\operatorname{Im}(\rho_m(x)+tI).
\]
Choosing $t$ successively as the negatives of all eigenvalues shows that every eigen-coordinate of
$\Psi(\rho_m(x))$ is zero. Hence $\Psi$ vanishes on every semisimple element of
$\rho_m(\mathfrak{sl}_2)$. Such elements span $\rho_m(\mathfrak{sl}_2)$, so $\Psi=0$.
\end{proof}

\begin{lemma}\label{lem:intersection-images}
For the non-trivial irreducible module $U_m$,
\[
\bigcap_{x\in\mathfrak{sl}_2}\rho_m(x)U_m=\{0\}.
\]
\end{lemma}
\begin{proof}
It is enough to intersect the images of the conjugates of $e$.  For
$x(t)=e-th-t^2f=\operatorname{Ad}(\exp(tf))e$, a left annihilator of
$\rho_m(x(t))U_m$ is the functional
\[
\varphi_t(u_j)=(-t)^{m-j}\qquad(0\le j\le m).
\]
As $t$ varies, these functionals span $U_m^*$. Hence the intersection of their kernels, and therefore the intersection of the images, is zero.
\end{proof}

After the preceding reductions, the only possible new local maps are local coboundaries.

\begin{proposition}\label{prop:pure-local-part}
After subtracting a derivation, every local derivation on $\g_\ell(d)$ can be written in the form
\begin{equation}\label{eq:pure-local-form}
\Delta(V_{\ell,d})=0,
\qquad
\Delta(x)=F(x)\in V_{\ell,d}\quad(x\in S),
\end{equation}
where $F:S\to V_{\ell,d}$ is linear and satisfies
\begin{equation}\label{eq:range-local-condition}
F(x)\in\rho(x)V_{\ell,d}\qquad(x\in S).
\end{equation}
Conversely, every linear map satisfying \eqref{eq:range-local-condition} defines a local derivation by \eqref{eq:pure-local-form}.
\end{proposition}
\begin{proof}
Since $V_{\ell,d}$ is characteristic, every derivation preserves it; hence every local derivation satisfies $\Delta(V_{\ell,d})\subseteq V_{\ell,d}$ and induces a linear map on the quotient. Let $\pi:S\ltimes V_{\ell,d}\to S$ be the quotient map.  For $x\in S$, a pointwise derivation realizing $\Delta(x)$ induces a derivation of $S$ realizing $\pi\Delta(x)$.  Thus $\pi\Delta|_S$ is a local derivation of $S$, and Lemma~\ref{lem:local-on-S} shows that it is an ordinary derivation.

When $d\ne2$, this quotient derivation is inner and therefore lifts to an inner derivation of $\g_\ell(d)$.  When $d=2$, Lemma~\ref{lem:local-on-S} gives
\[
\pi\Delta|_S=\ad a+\mu\vartheta,
\qquad
\vartheta|_{\mathfrak{sl}_2}=0,
\quad
\vartheta(s)=s.
\]
The scaling direction $\vartheta$ does not lift to $\g_\ell(2)$, but it cannot occur here.  Indeed, locality at $s$ gives
\[
\pi\Delta(s)=\pi D_s(s)=0,
\]
because Theorem~\ref{thm:der0} shows that every derivation of $\g_\ell(2)$ induces an inner derivation on the quotient and hence kills the central class of $s$.  Therefore $\mu=0$.  In every dimension the quotient derivation is consequently liftable; subtracting such a lift, we may assume that the induced map on the reductive quotient is zero.

By Proposition~\ref{prop:module-reflexive}, the restriction to $V_{\ell,d}$ has the form
$\rho(b)+\lambda I$. Lemma~\ref{lem:noncentral-compatibility} shows that $b$ is central in $S$. Therefore $\rho(b)+\lambda I$ is itself the restriction of a derivation which is zero on the quotient; subtracting it gives \eqref{eq:pure-local-form}.

For $x\in S$, the pointwise derivation attached to $x$ has zero reductive component at $x$, and hence its radical contribution is $-\rho(x)u_x$ for some $u_x\in V$. This proves \eqref{eq:range-local-condition}.

Conversely, if \eqref{eq:range-local-condition} holds, choose $u_x$ with
$F(x)=-\rho(x)u_x$. For any $x+v\in S\ltimes V$ one has
\[
\operatorname{ad}(u_x)(x+v)=-\rho(x)u_x=F(x)=\Delta(x+v),
\]
because $V$ is abelian. Thus $\Delta$ is local.
\end{proof}

\begin{lemma}\label{lem:two-plane-rigidity}
Assume $d\ge2$ and let $F:S\to U_m\otimes W_d$ be linear with
\[
F(x)\in\rho(x)(U_m\otimes W_d)\qquad(x\in S).
\]
Fix $1\le i<j\le d$, put
$E_{ij}=\operatorname{span}\{\varepsilon_i,\varepsilon_j\}$, and let
$\pi_{ij}:U_m\otimes W_d\to U_m\otimes E_{ij}$ be the natural projection.  Then there exists
$u_{ij}\in U_m\otimes E_{ij}$ such that
\begin{equation}\label{eq:two-plane-coboundary}
\pi_{ij}F(a+t s_{ij})
=\rho(a+t s_{ij})u_{ij}
\qquad(a\in\mathfrak{sl}_2,\ t\in\C).
\end{equation}
\end{lemma}
\begin{proof}
The operator $\rho(a+t s_{ij})$ preserves the direct sum
\[
U_m\otimes W_d=(U_m\otimes E_{ij})\oplus(U_m\otimes E_{ij}^{\perp}).
\]
Choose eigenvectors
$w_\pm=\varepsilon_i\pm\mathrm i\varepsilon_j$ of $s_{ij}$, with eigenvalues
$\pm c$, where $c\neq0$.  Write
\[
\pi_{ij}F(a+t s_{ij})
=\Phi_+(a,t)\otimes w_+ +\Phi_-(a,t)\otimes w_-.
\]
Range locality gives
\[
\Phi_\pm(a,t)\in
\bigl(\rho_m(a)\pm ctI\bigr)U_m.
\]
Since $(a,t)\mapsto\rho_m(a)\pm ctI$ identifies
$\mathfrak{sl}_2\oplus\C$ with
$\mathcal A_m=\rho_m(\mathfrak{sl}_2)+\C I$, Lemma~\ref{lem:range-local-gl2} yields vectors
$u_\pm\in U_m$ such that
\[
\Phi_\pm(a,t)=\bigl(\rho_m(a)\pm ctI\bigr)u_\pm.
\]
Taking
$u_{ij}=u_+\otimes w_++u_-\otimes w_-$ proves \eqref{eq:two-plane-coboundary}.
\end{proof}

\begin{theorem}\label{thm:local-noncentral-dge2}
If $d\ge2$, then
\begin{equation}\label{eq:locder-noncentral-dge2}
\LocDer\g_\ell(d)=\Der\g_\ell(d)
\qquad\left(\ell\in\tfrac12\mathbb N\right).
\end{equation}
\end{theorem}
\begin{proof}
By Proposition~\ref{prop:pure-local-part}, it is enough to classify the linear maps
$F:S\to U_m\otimes W_d$ satisfying the range-local condition.

First let $d=2$.  Then
$S=\mathfrak{sl}_2\oplus\C s_{12}$, so Lemma~\ref{lem:two-plane-rigidity} applies to the whole space and gives
\[
F(x)=\rho(x)u_{12}\qquad(x\in S).
\]
Thus $F$ is an ordinary coboundary.

Assume now $d\ge3$.  Write
\[
F(a)=\sum_{r=1}^d F_r(a)\otimes\varepsilon_r
\qquad(a\in\mathfrak{sl}_2).
\]
For each pair $i<j$, Lemma~\ref{lem:two-plane-rigidity} gives
$u_{ij}\in U_m\otimes E_{ij}$ with
\[
\pi_{ij}F(a)=\rho(a)u_{ij}
\qquad(a\in\mathfrak{sl}_2).
\]
Write the $\varepsilon_i$-component of $u_{ij}$ as $u_i^{(j)}$.  If $j$ and $k$ are distinct from $i$, then
\[
\rho_m(a)u_i^{(j)}=F_i(a)=\rho_m(a)u_i^{(k)}
\qquad(a\in\mathfrak{sl}_2).
\]
Hence $u_i^{(j)}-u_i^{(k)}$ is annihilated by all of $\mathfrak{sl}_2$.  The irreducible non-trivial module $U_m$ has no invariant vector, so these vectors coincide.  Denote the common value by $u_i$ and put
\[
u_0=\sum_{i=1}^d u_i\otimes\varepsilon_i.
\]
Then
\[
F(a)=\rho(a)u_0\qquad(a\in\mathfrak{sl}_2).
\]
Subtract the coboundary $x\mapsto\rho(x)u_0$.  We may therefore assume
$F(\mathfrak{sl}_2)=0$.

Fix $i<j$.  Applying Lemma~\ref{lem:two-plane-rigidity} to the normalized map gives a vector
$u_{ij}$ satisfying \eqref{eq:two-plane-coboundary}.  Setting $t=0$ yields
$\rho(a)u_{ij}=0$ for every $a\in\mathfrak{sl}_2$, hence $u_{ij}=0$.  Consequently
\[
\pi_{ij}F(s_{ij})=0.
\]
For $k\notin\{i,j\}$, take the $\varepsilon_k$-component in the range-local relation for
$a+t s_{ij}$. Since $s_{ij}\varepsilon_k=0$ and $F(a)=0$, one obtains
\[
F_k(s_{ij})\in\rho_m(a)U_m
\qquad(a\in\mathfrak{sl}_2).
\]
Lemma~\ref{lem:intersection-images} therefore gives $F_k(s_{ij})=0$.  Hence
$F(s_{ij})=0$.  The generators $s_{ij}$ span $\mathfrak{so}_d$, so $F=0$ after subtraction of the coboundary.  Proposition~\ref{prop:pure-local-part} completes the proof.
\end{proof}

It remains to classify the one-dimensional spatial case.  Then $S=\mathfrak{sl}_2$ and
$V_{\ell,1}=U_m$.

\begin{proposition}\label{prop:sl2-local-coboundaries}
Let
\[
\mathcal R_m=
\{F\in\operatorname{Hom}(\mathfrak{sl}_2,U_m):F(x)\in\rho_m(x)U_m\text{ for all }x\}.
\]
Then
\begin{equation}\label{eq:R-m-classification}
\mathcal R_m\cong
\begin{cases}
U_m, & m\text{ even},\\
U_m, & m=1,\\
U_m\oplus U_{m-2}, & m\ge3\text{ odd}.
\end{cases}
\end{equation}
The summand $U_m$ is exactly the space of ordinary coboundaries
$x\mapsto\rho_m(x)u$.
\end{proposition}
\begin{proof}
The space $\mathcal R_m$ is invariant under the natural $\mathfrak{sl}_2$-action on
$\operatorname{Hom}(\mathfrak{sl}_2,U_m)$. By the Clebsch--Gordan decomposition,
\begin{equation}\label{eq:CG}
\operatorname{Hom}(\mathfrak{sl}_2,U_m)
\cong U_2\otimes U_m
\cong U_{m+2}\oplus U_m\oplus U_{m-2},
\end{equation}
where the last summand is omitted for $m=1$. The coboundaries form the middle summand $U_m$.

A highest weight vector of the $U_{m+2}$-summand is the map
\[
F_+(e)=F_+(h)=0,
\qquad F_+(f)=u_0.
\]
It is not range-local because $u_0\notin\rho_m(f)U_m$. Hence $U_{m+2}$ is excluded.

For $m\ge2$, a highest weight vector of the $U_{m-2}$-summand is
\begin{equation}\label{eq:Psi0}
\Psi_0(e)=u_0,
\qquad
\Psi_0(h)=-2u_1,
\qquad
\Psi_0(f)=-u_2.
\end{equation}
If $m$ is even, put $x=e+f$. The operator $\rho_m(x)$ has one-dimensional cokernel. A left null vector
$\phi$ may be normalized by
\[
\phi(u_0)=1,
\qquad
\phi(u_2)=-\frac1{m-1},
\]
with the remaining even coefficients determined recursively from
$k\phi(u_{k-1})+(m-k)\phi(u_{k+1})=0$ and all odd coefficients zero. Consequently
\[
\phi(\Psi_0(e+f))
=\phi(u_0-u_2)=1+\frac1{m-1}\ne0.
\]
Thus $\Psi_0(e+f)\notin\rho_m(e+f)U_m$, and the $U_{m-2}$-summand is excluded when $m$ is even.

Now let $m\ge3$ be odd. Every non-zero semisimple element of $\mathfrak{sl}_2$ acts invertibly on $U_m$, so only nilpotent elements need to be checked. Up to conjugacy and scalar, they are
\[
x(t)=e-th-t^2f.
\]
A left annihilator of $\rho_m(x(t))U_m$ is the functional from Lemma~\ref{lem:intersection-images},
$\varphi_t(u_j)=(-t)^{m-j}$. Using \eqref{eq:Psi0},
\[
\Psi_0(x(t))=u_0+2t u_1+t^2u_2
\]
and therefore
\[
\varphi_t(\Psi_0(x(t)))
=(-t)^m+2t(-t)^{m-1}+t^2(-t)^{m-2}=0.
\]
The limiting nilpotent direction $f$ is immediate as well. Hence $\Psi_0\in\mathcal R_m$ for odd $m$. By invariance, the entire $U_{m-2}$-summand is contained in $\mathcal R_m$. Together with \eqref{eq:CG}, this proves \eqref{eq:R-m-classification}.
\end{proof}

\begin{theorem}\label{thm:local-noncentral}
Let $m=2\ell$.
\begin{enumerate}[label=\textup{(\roman*)}]
\item If $d\ge2$, then
\[
\LocDer\g_\ell(d)=\Der\g_\ell(d).
\]
\item If $d=1$ and either $m$ is even or $m=1$, then
\[
\LocDer\g_\ell(1)=\Der\g_\ell(1).
\]
\item If $d=1$ and $m\ge3$ is odd, then
\begin{equation}\label{eq:locder-exception}
\LocDer\g_\ell(1)
=\Der\g_\ell(1)\oplus\mathcal P_m
\end{equation}
as vector spaces, where $\mathcal P_m\cong U_{m-2}$ and
\begin{equation}\label{eq:dim-pure}
\dim\mathcal P_m=m-1=2\ell-1.
\end{equation}
In particular, $\g_\ell(1)$ admits pure local derivations exactly for
$\ell=\frac32,\frac52,\frac72,\ldots$.
\end{enumerate}
\end{theorem}
\begin{proof}
Part (i) is Theorem~\ref{thm:local-noncentral-dge2}. For $d=1$, Proposition~\ref{prop:pure-local-part} identifies the quotient of local derivations by ordinary derivations with
$\mathcal R_m/U_m$. Proposition~\ref{prop:sl2-local-coboundaries} gives parts (ii) and (iii), including the dimension formula.
\end{proof}

\begin{corollary}\label{cor:explicit-pure}
Assume $d=1$ and $m=2\ell\ge3$ is odd. With
$u_k=P_1^{(k)}$, the linear map
\[
\Psi_0(e)=u_0,
\qquad
\Psi_0(h)=-2u_1,
\qquad
\Psi_0(f)=-u_2,
\qquad
\Psi_0(u_k)=0
\]
is a local derivation which is not a derivation. Its $\mathfrak{sl}_2$-span is the full pure-local space $\mathcal P_m$ of dimension $m-1$.
\end{corollary}
\begin{proof}
Locality follows from the odd-$m$ part of Proposition~\ref{prop:sl2-local-coboundaries} and the converse statement in Proposition~\ref{prop:pure-local-part}.  It is not an ordinary derivation. Indeed, using $[h,e]=2e$ and \eqref{eq:Psi0},
\[
\Psi_0([h,e])=2u_0,
\]
whereas
\[
[\Psi_0(h),e]+[h,\Psi_0(e)]
=[-2u_1,e]+[h,u_0]
=(m+2)u_0.
\]
Since $m\ge3$, the derivation identity fails.
\end{proof}

\begin{proposition}\label{prop:first-exception}
Let $d=1$ and $\ell=\frac32$, so that $m=3$ and
\[
\g_{3/2}(1)=\mathfrak{sl}_2\ltimes U_3,
\qquad
U_3=\langle u_0,u_1,u_2,u_3\rangle,
\]
with
\[
[e,u_k]=k u_{k-1},\qquad
[f,u_k]=(3-k)u_{k+1},\qquad
[h,u_k]=(3-2k)u_k.
\]
Define a linear operator $\Psi_0$ by
\begin{equation}\label{eq:Psi0-m3}
\Psi_0(e)=u_0,\qquad
\Psi_0(h)=-2u_1,\qquad
\Psi_0(f)=-u_2,\qquad
\Psi_0(U_3)=0.
\end{equation}
Then $\Psi_0$ is a local derivation but not a derivation. Consequently,
\[
\LocDer\g_{3/2}(1)\ne\Der\g_{3/2}(1),
\qquad
\dim\bigl(\LocDer\g_{3/2}(1)/\Der\g_{3/2}(1)\bigr)=2.
\]
A second generator of the quotient may be chosen as
\begin{equation}\label{eq:Psi1-m3}
\Psi_1(e)=u_1,\qquad
\Psi_1(h)=-2u_2,\qquad
\Psi_1(f)=-u_3,\qquad
\Psi_1(U_3)=0.
\end{equation}
The classes of $\Psi_0$ and $\Psi_1$ span the irreducible $U_1$ pure-local component.
\end{proposition}
\begin{proof}
We first verify locality of $\Psi_0$ directly.  Let
\[
x=s+v,\qquad s\in\mathfrak{sl}_2,\quad v\in U_3.
\]
Since $\Psi_0(v)=0$ and $U_3$ is abelian, it is enough to prove that for every
$s\in\mathfrak{sl}_2$ there exists $w_s\in U_3$ such that
\begin{equation}\label{eq:local-ws}
\Psi_0(s)=[w_s,s]=-\rho_3(s)w_s.
\end{equation}
Indeed, then the inner derivation $D_x=\ad(w_s)$ satisfies
\[
D_x(x)=[w_s,s+v]=[w_s,s]=\Psi_0(s)=\Psi_0(x).
\]
If $s=0$, take $D_x=0$.

Suppose first that $s\ne0$ is semisimple.  On the irreducible module $U_3$ the
weights of $s$ are proportional to $3,1,-1,-3$, hence $\rho_3(s)$ is invertible.
Therefore \eqref{eq:local-ws} has a solution $w_s$.

It remains to consider nilpotent elements.  Apart from the limiting direction $f$,
every nilpotent line can be represented by
\[
s_t=e-th-t^2f\qquad(t\in\C).
\]
For $t\ne0$, put
\[
w_t=\frac{1}{2t}u_0+\frac12u_1.
\]
Using the three module relations above, one obtains
\[
-\rho_3(s_t)w_t
 =u_0+2tu_1+t^2u_2
 =\Psi_0(s_t).
\]
For $t=0$ we take $w_0=-u_1$, since
\[
-[e,-u_1]=u_0=\Psi_0(e).
\]
For the remaining nilpotent direction $s=f$, take $w_f=\frac12u_1$; then
\[
-[f,\tfrac12u_1]=-u_2=\Psi_0(f).
\]
Thus $\Psi_0$ is local.

The map $\Psi_0$ is not a derivation.  Indeed, from $[h,e]=2e$ we have
\[
\Psi_0([h,e])=2u_0,
\]
whereas
\[
[\Psi_0(h),e]+[h,\Psi_0(e)]
=[-2u_1,e]+[h,u_0]
=2u_0+3u_0=5u_0.
\]
Hence the derivation identity fails.

Finally, Proposition~\ref{prop:sl2-local-coboundaries} identifies the pure-local quotient
with the irreducible module $U_{m-2}=U_1$, which has dimension two.  The second map
$\Psi_1$ in \eqref{eq:Psi1-m3} is obtained from $\Psi_0$ by the natural lowering action of
$f$ on $\operatorname{Hom}(\mathfrak{sl}_2,U_3)$:
\[
(f\cdot\Psi_0)(x)=\rho_3(f)\Psi_0(x)-\Psi_0([f,x]).
\]
A direct calculation gives
\[
(f\cdot\Psi_0)(e)=u_1,\qquad
(f\cdot\Psi_0)(h)=-2u_2,\qquad
(f\cdot\Psi_0)(f)=-u_3.
\]
Thus $\Psi_1=f\cdot\Psi_0$.  Since the range-local space is an $\mathfrak{sl}_2$-submodule, $\Psi_1$ is local, and the two non-zero weight vectors $\Psi_0,\Psi_1$ span the irreducible $U_1$ quotient.
\end{proof}

\begin{remark}\label{rem:second-exception}
For $d=1$ and $\ell=\frac52$, one has $m=5$ and the pure-local component is
$U_{m-2}=U_3$.  Hence
\[
\dim\bigl(\LocDer\g_{5/2}(1)/\Der\g_{5/2}(1)\bigr)=4.
\]
\end{remark}

\begin{remark}[Low-weight consistency check]\label{rem:low-weight-check}
Proposition~\ref{prop:sl2-local-coboundaries} gives, for the first five strings,
\[
\dim\mathcal R_1=2,\qquad
\dim\mathcal R_2=3,\qquad
\dim\mathcal R_3=6,\qquad
\dim\mathcal R_4=5,\qquad
\dim\mathcal R_5=10.
\]
After quotienting by the $(m+1)$-dimensional coboundary space $U_m$, the corresponding pure-local dimensions are
\[
0,\ 0,\ 2,\ 0,\ 4.
\]
Thus the explicit $m=3$ calculation above is the first member of the general odd-string family, while the neighboring even strings have no pure local component.
\end{remark}

\begin{remark}[Effect of the central cocycle]\label{rem:central-rigidification}
The exceptional family in Theorem~\ref{thm:local-noncentral} occurs precisely because the radical of
$\g_\ell(1)$ is abelian.  In the mass central extension, the same odd values of $m$ carry the non-degenerate complementary pairing
$[p_r,p_{m-r}]\ne0$.  The central coefficient comparisons used below eliminate the entire
$U_{m-2}$ pure-local summand.  Thus the mass cocycle rigidifies the local-derivation problem.
\end{remark}

\subsection{Local derivations on the mass central extension}\label{subsec:local-mass}
We next study the mass extension. Here the Heisenberg pairing between complementary levels introduces additional coefficient constraints, and the main task is to show that these constraints eliminate every pure local degree of freedom left by the unextended algebra.

Put $m=2\ell$; in the mass case $m$ is odd. It is convenient to use
\[
h=D,\qquad e=-H,\qquad f=C,\qquad z=M,
\qquad p_{i,k}=P_i^{(k)}\quad(1\le i\le d,\ 0\le k\le m).
\]
Then
\begin{align*}
[h,e]&=2e,&[h,f]&=-2f,&[e,f]&=h,\\
[h,p_{i,k}]&=(m-2k)p_{i,k},&
[e,p_{i,k}]&=k p_{i,k-1},&
[f,p_{i,k}]&=(m-k)p_{i,k+1},
\end{align*}
and
\begin{equation}\label{eq:mass-pairing-local}
[p_{i,r},p_{j,s}]=\delta_{ij}\delta_{r+s,m}I_r z,
\qquad I_{m-r}=-I_r .
\end{equation}
The rotation action is the one in \eqref{eq:noncentral-table}.

For $d\neq2$, every pointwise derivation has the form
\begin{equation}\label{eq:localM1}
D_x=\ad a_x+\lambda_x\delta_M,
\end{equation}
whereas for $d=2$ one has
\begin{equation}\label{eq:localM2}
D_x=\ad a_x+\lambda_x\delta_M+\mu_x\eta_M.
\end{equation}

\begin{lemma}\label{lem:localM-support}
Let $\Delta$ be a local derivation on $\widehat{\g}^{\,M}_\ell(d)$. Then
\[
\Delta(z)\in\C z,
\]
and for every $i$ and $0\le k\le m$,
\begin{equation}\label{eq:localM-Psupport}
\Delta(p_{i,k})\in
\operatorname{span}\Bigl(
\{p_{i,k-1},p_{i,k+1},z\}\cup\{p_{j,k}:1\le j\le d\}
\Bigr),
\end{equation}
where the terms with indices $-1$ and $m+1$ are omitted. The restrictions on
$\Delta(h),\Delta(e),\Delta(f)$ are the corresponding $h,e,f$ versions of
Lemma~\ref{lem:local0-support}.
\end{lemma}
\begin{proof}
The center is characteristic and the inner part of a pointwise derivation vanishes on $z$, while
$\delta_M(z)=2z$ and $\eta_M(z)=0$. Hence $\Delta(z)\in\C z$.
For $p_{i,k}$, the $e$- and $f$-parts of $a_x$ change the level by one, the $h$-part and
$\delta_M$ preserve the level, and the rotation part changes only the spatial index. The only extra
term comes from a complementary $p_{i,m-k}$-component of $a_x$, which contributes a central
multiple of $z$ through \eqref{eq:mass-pairing-local}.
\end{proof}

\begin{lemma}\label{lem:mass-normalization}
Let $\Delta\in\LocDer\widehat{\g}^{\,M}_\ell(d)$. There exists
$D_0\in\Der\widehat{\g}^{\,M}_\ell(d)$ such that, for
$\Delta_0=\Delta-D_0$,
\begin{equation}\label{eq:mass-normalized}
\pi\Delta_0(x)=0\quad(x\in\slc\oplus\sod),
\qquad
\Delta_0(h)=\Delta_0(z)=0,
\end{equation}
where $\pi$ is the quotient map modulo the nilradical
$V_{\ell,d}\oplus\C z$. For $d=2$ the same statement holds with
$\sod=\C s$, $s=-\mathrm iM_{12}$.
\end{lemma}
\begin{proof}
The nilradical $V_{\ell,d}\oplus\C z$ is characteristic. Therefore every pointwise derivation preserves it, so $\Delta$ induces a well-defined linear local derivation on the reductive quotient. On the semisimple part of that quotient, the induced map is a local derivation and therefore an inner derivation. In the case $d=2$, every derivation of the full algebra induces zero on
the one-dimensional factor $\C s$, so the induced local map is zero there as well. Subtracting the
corresponding inner derivation gives
\[
\pi\Delta(x)=0\qquad(x\in\slc\oplus\sod).
\]

Choose a pointwise derivation $D_{h+z}$ such that
$\Delta(h+z)=D_{h+z}(h+z)$. Write the reductive component of the inner element of
$D_{h+z}$ as $b$. Since the left-hand side has no reductive component,
$[b,h]=0$. Hence $D_{h+z}-\ad b$ has the same value at $h+z$ and induces zero on the
reductive quotient. Subtracting this derivation preserves the first normalization and gives
$\Delta(h+z)=0$. Lemma~\ref{lem:localM-support} gives
$\Delta(h)\in V_{\ell,d}$ and $\Delta(z)\in\C z$, so linear independence yields
$\Delta(h)=\Delta(z)=0$.
\end{proof}

For each spatial index $i$, let
\[
L_i=\operatorname{span}\{e,f,h,z,p_{i,0},\ldots,p_{i,m}\}.
\]
This is the one-string mass conformal Galilei algebra.

\begin{lemma}\label{lem:mass-low-parameters}
For the one-string mass algebra one has
\[
\LocDer\widehat{\g}^{\,M}_{1/2}(1)=\Der\widehat{\g}^{\,M}_{1/2}(1),
\qquad
\LocDer\widehat{\g}^{\,M}_{3/2}(1)=\Der\widehat{\g}^{\,M}_{3/2}(1).
\]
\end{lemma}
\begin{proof}
Let $\nabla$ be a local derivation on the one-string algebra.  By the $d=1$ case of
Lemma~\ref{lem:mass-normalization}, after subtracting one global derivation we may assume
\begin{equation}\label{eq:low-mass-normalization}
\pi\nabla(x)=0\quad(x\in\mathfrak{sl}_2),
\qquad
\nabla(h)=\nabla(z)=0.
\end{equation}
It is enough to prove that such a normalized map is zero.  At each test element below we write
its pointwise derivation in the form
\[
D_x=\ad y_x+\lambda_x\delta_M
\]
and compare the coefficients in the fixed basis.  Only the defining relations above are used.

\smallskip
\noindent\emph{Case $\ell=\frac12$ ($m=1$).}
Here $I_0=-1$.  Lemma~\ref{lem:localM-support} and
\eqref{eq:low-mass-normalization} give
\begin{align*}
\nabla(e)&=\alpha p_0, & \nabla(f)&=\beta p_1,\\
\nabla(p_0)&=a p_0+b p_1+r z,
&\nabla(p_1)&=c p_0+d p_1+s z.
\end{align*}
Successive applications of locality give the following relations:
\begin{center}
\begin{tabular}{c|c}
 test element & relation obtained by coefficient comparison \\
\hline
$e-p_0-z$ & $b=0$\\
$f-p_1-z$ & $c=0$\\
$h-p_0$, $h+p_0$ & $-b-r=0$, $-b+r=0$\\
$h-f-p_1$, $h-f+p_1$ & $-c-s=0$, $-c+s=0$\\
$e-h-f-z$, $e+h-f-z$ & $-\alpha-\beta=0$, $\alpha-\beta=0$\\
$h+p_0+p_1+z$ & $-a-b-c-d+r+s=0$\\
$e+f+p_0+p_1$ & $a+\alpha-b-\beta+c-d+r+s=0$.
\end{tabular}
\end{center}
Hence
\[
b=c=r=s=\alpha=\beta=0,
\qquad a+d=0,
\qquad a-d=0,
\]
so $a=d=0$ and therefore $\nabla=0$.

\smallskip
\noindent\emph{Case $\ell=\frac32$ ($m=3$).}
Now
\[
I_0=6,\qquad I_1=-2,\qquad I_2=2,\qquad I_3=-6.
\]
By Lemma~\ref{lem:localM-support}, write
\begin{align*}
\nabla(e)&=E_0p_0+E_1p_1+E_2p_2,\\
\nabla(f)&=F_1p_1+F_2p_2+F_3p_3,\\
\nabla(p_0)&=A_{00}p_0+A_{01}p_1+A_{0z}z,\\
\nabla(p_1)&=A_{10}p_0+A_{11}p_1+A_{12}p_2+A_{1z}z,\\
\nabla(p_2)&=A_{21}p_1+A_{22}p_2+A_{23}p_3+A_{2z}z,\\
\nabla(p_3)&=A_{32}p_2+A_{33}p_3+A_{3z}z.
\end{align*}
We record the coefficient comparisons in a reduction order.  Locality at $e+p_2$ and
$f+p_1$ gives
\[
A_{23}=0,\qquad A_{2z}=2E_0,
\qquad
A_{10}=0,\qquad A_{1z}=-2F_3.
\]
The pairs $e+p_0$, $h+p_0$ and $h\pm p_1$, $h\pm p_2$ yield
\[
A_{0z}=E_2=A_{12}=A_{1z}=F_3=A_{21}=A_{2z}=E_0=0.
\]
Next, locality at $h+p_3$ and $f+p_3$ gives
\[
A_{3z}=F_1=0.
\]
Using $e+h-f$ and $e-h-f$ we then obtain
\[
E_1=F_2=0,
\]
so $\nabla(e)=\nabla(f)=0$.  The test elements
$h+p_0+p_2$ and $h+p_1+p_3$ give
\[
A_{01}=A_{32}=0.
\]
Finally, applying locality to
\[
e+h+p_0+p_3,
\qquad e+h+p_1-p_2,
\qquad e+h+p_1-p_3
\]
gives, after the preceding vanishing relations,
\[
A_{00}=A_{33},
\qquad
A_{11}=A_{22},
\qquad
A_{11}=A_{33}.
\]
Thus there is a scalar $t$ such that
\[
\nabla(p_k)=t p_k\qquad(0\le k\le3),
\qquad
\nabla(e)=\nabla(f)=\nabla(h)=\nabla(z)=0.
\]
It remains to eliminate this last scalar.  Apply locality to
\[
x=h+p_0+p_3-2z.
\]
Writing $D_x=\ad y_x+\lambda_x\delta_M$ and comparing the $p_0,p_3,z$ coefficients gives
\[
2A_{00}+A_{0z}+2A_{33}+A_{3z}=0.
\]
All central coefficients have already vanished, hence $4t=0$ and $t=0$.
Therefore $\nabla=0$ also for $m=3$.  Restoring the derivation subtracted in
\eqref{eq:low-mass-normalization} proves both assertions.
\end{proof}

\begin{lemma}\label{lem:mass-one-string}
Assume that $\Delta$ satisfies \eqref{eq:mass-normalized}. Then
\begin{equation}\label{eq:mass-string-reduced}
\Delta(e)=\Delta(f)=0,
\qquad
\Delta(p_{i,k})=\sum_{\substack{1\le j\le d\\j\ne i}}
a_{ji}^{(k)}p_{j,k}
\quad(1\le i\le d,\ 0\le k\le m)
\end{equation}
for suitable scalars $a_{ji}^{(k)}$.
\end{lemma}
\begin{proof}
Let $\pi_i$ be the projection onto $L_i$ along the other spatial strings and the rotation part.
If $x\in L_i$ and
\[
D_x=\ad a_x+\lambda_x\delta_M
\quad\text{or}\quad
D_x=\ad a_x+\lambda_x\delta_M+\mu_x\eta_M,
\]
then the rotation component of $a_x$ sends $p_{i,k}$ into the other spatial strings, and the
components $p_{j,r}$ with $j\ne i$ have zero bracket with the $i$-th Heisenberg string. Consequently
$\pi_iD_x(x)$ is the value at $x$ of a derivation of $L_i$. Thus
$\pi_i\Delta|_{L_i}$ is a local derivation of the one-string mass algebra.

For $\ell\notin\{\frac12,\frac32\}$, the one-string rigidity theorem is proved in
\cite{AYCG}; the two missing low parameters are established directly in
Lemma~\ref{lem:mass-low-parameters}.  Hence every local derivation on $L_i$ is a derivation for
all admissible half-integers $\ell$.  Under the normalization \eqref{eq:mass-normalized}, this derivation induces zero
on $\mathfrak{sl}_2$ and vanishes at both $h$ and $z$. Write it as
$\ad u+\lambda\delta_M$. The quotient condition removes the $\mathfrak{sl}_2$-part of $u$, while
$0=(\ad u)(h)$ forces the $V$-part of $u$ to vanish because the $h$-weights
$m,m-2,\ldots,-m$ are all non-zero. Finally,
$0=\lambda\delta_M(z)=2\lambda z$ gives $\lambda=0$. Hence
$\pi_i\Delta|_{L_i}=0$.
Since this holds for every $i$, all same-string adjacent, diagonal and central terms disappear from
\eqref{eq:localM-Psupport}, and the same projections give $\Delta(e)=\Delta(f)=0$.
This proves \eqref{eq:mass-string-reduced}.
\end{proof}

\begin{lemma}\label{lem:mass-spatial}
Under the assumptions of Lemma~\ref{lem:mass-one-string}, there exists a single
skew-symmetric matrix $A=(a_{ji})\in\mathfrak{so}_d(\C)$ such that
\begin{equation}\label{eq:mass-spatial-fixed}
a_{ji}^{(k)}=a_{ji}\quad\text{for all }i,j,k.
\end{equation}
If, in addition, \eqref{eq:mass-normalized} holds, then $A=0$ for $d\ne2$; for $d=2$ it is a scalar
multiple of the standard generator.
\end{lemma}
\begin{proof}
Fix $r<t$ and a level $k$. Apply locality to
\[
x=p_{r,k}+\mathrm i p_{t,k}.
\]
Let $Q=(q_{ji})\in\mathfrak{so}_d$ be the rotation component of the attached pointwise inner
derivation and put
\[
c=(m-2k)\alpha_h+\lambda
\]
for the common same-level scalar contribution of its $h$- and $\delta_M$-parts.
Comparing the $p_{r,k}$- and $p_{t,k}$-coefficients gives
\[
\mathrm i\,a_{rt}^{(k)}=c+\mathrm i q_{rt},
\qquad
a_{tr}^{(k)}=-q_{rt}+\mathrm i c.
\]
Eliminating $c$ and $q_{rt}$ yields
$a_{tr}^{(k)}=-a_{rt}^{(k)}$.
Thus the matrix $A^{(k)}=(a_{ji}^{(k)})$ is skew-symmetric.

Next use
\[
x=p_{r,k}+p_{r,k+1}\qquad(0\le k<m).
\]
For every $t\ne r$, the coefficients of $p_{t,k}$ and $p_{t,k+1}$ are produced by one and the same
rotation coefficient $q_{tr}$; the $\mathfrak{sl}_2$- and scalar parts do not change the spatial index.
Hence $a_{tr}^{(k)}=a_{tr}^{(k+1)}$. Therefore
$A^{(0)}=\cdots=A^{(m)}=:A$.

Assume now $d\ge3$ and take a rotation generator $R=M_{ab}$. Let
$P_{ab}=\operatorname{span}\{\varepsilon_a,\varepsilon_b\}$ and take
$w\in P_{ab}$ and
\[
v_k(w)=\sum_{r=1}^d w_r p_{r,k}.
\]
For each non-zero $t$, apply locality to $R+t\,v_k(w)$. The quotient part of the left-hand side is
zero, so the rotation component $Q_t$ of the attached pointwise derivation satisfies
$[Q_t,R]=0$. Hence $Q_t$ preserves $P_{ab}$. The same is true for
$R(V)$, while the $\mathfrak{sl}_2$- and scalar parts preserve the spatial vector $w$.
Consequently the level-$k$ spatial component of
$\Delta(R)+t\,A w$ belongs to $P_{ab}$ for every non-zero $t$. Taking two different values of $t$
and subtracting shows $A(P_{ab})\subseteq P_{ab}$.
This holds for every coordinate two-plane. For $d\ge3$,
\[
A\varepsilon_a\in\bigcap_{b\ne a}P_{ab}=\C\varepsilon_a,
\]
and skew-symmetry then gives $A\varepsilon_a=0$ for every $a$. Thus $A=0$.
For $d=1$ there is no rotation part. For $d=2$ every skew matrix is a scalar multiple of the
standard generator $s=-\mathrm iM_{12}$.
\end{proof}

\begin{lemma}\label{lem:mass-rotation-images}
Assume \eqref{eq:mass-normalized}.
\begin{enumerate}[label=\textup{(\roman*)}]
\item If $d\ge3$, then $\Delta$ vanishes on $\sod$.
\item If $d=2$, then there are $a,b\in\C$ such that
\[
\Delta|_{V_{\ell,2}}=a\,\ad s|_{V_{\ell,2}},
\qquad
\Delta(s)=bz.
\]
\end{enumerate}
\end{lemma}
\begin{proof}
First suppose $d\ge3$. By Lemma~\ref{lem:mass-spatial}, $\Delta(V_{\ell,d})=0$.
Fix $R=M_{ab}$. Since the quotient action of $\Delta$ is zero, locality at $R$ gives
\[
\Delta(R)\in [V_{\ell,d},R]\subset
\operatorname{span}\{p_{a,k},p_{b,k}:0\le k\le m\},
\]
and no central component can occur. Write
\[
\Delta(R)=\sum_{k=0}^m
\bigl(\alpha_kp_{a,k}+\beta_kp_{b,k}\bigr).
\]

Fix a target level $k$ and put $r=m-k$. Consider first
\[
y=R+p_{a,r}.
\]
Write the level-$k$ part of the radical component of the attached inner element as
\[
c_a p_{a,k}+c_b p_{b,k}+u',
\]
where $u'$ contains no $p_{a,k}$- or $p_{b,k}$-term. Since $k+r=m$, the only central contribution involving these two coefficients is
\[
[c_a p_{a,k},p_{a,r}]=c_a I_k z.
\]
The left-hand side $\Delta(R)$ has no central component, so $c_a=0$. Moreover,
\[
[c_a p_{a,k}+c_b p_{b,k},R]=c_a p_{b,k}-c_b p_{a,k}.
\]
If $k\notin\{r-1,r,r+1\}$, neither the $\mathfrak{sl}_2$-part nor the scalar part of the pointwise derivation acting on $p_{a,r}$ can produce a level-$k$ vector. Therefore the $p_{b,k}$-coordinate of $\Delta(R)$ is $c_a=0$, and hence $\beta_k=0$. Repeating the argument with
\[
y=R+p_{b,r}
\]
forces $c_b=0$ and then $\alpha_k=0$. Thus $\alpha_k=\beta_k=0$ away from the two middle levels.

It remains to treat the middle pair. Write $m=2s+1$. For $y=R+p_{a,s}$ the complementary level is $s+1$; the central coefficient forces the $p_{a,s+1}$-coefficient of the radical inner part to vanish. The $f$-term applied to $p_{a,s}$ stays in the $a$-string, so it cannot contribute to the $p_{b,s+1}$-coordinate. Hence $\beta_{s+1}=0$. The test vector $R+p_{b,s}$ similarly gives $\alpha_{s+1}=0$. Using $R+p_{a,s+1}$ and $R+p_{b,s+1}$, with the $e$-term now preserving the spatial index, gives $\beta_s=\alpha_s=0$. Consequently $\Delta(R)=0$. Since the $M_{ab}$ span $\sod$, part~(i) follows.

Now let $d=2$. Lemma~\ref{lem:mass-spatial} gives
$\Delta|_V=a\,\ad s|_V$. Subtract $a\ad s$ and keep the same notation for the resulting local
derivation. Then
\[
\Delta(e)=\Delta(f)=\Delta(h)=\Delta(z)=0,
\qquad
\Delta(V)=0.
\]
At $s$ itself, the pointwise formula gives
$\Delta(s)=w+bz$ with $w\in V$.
For every $t\in\C$, locality at $th+s$ implies that the $V$-part $w$ belongs to
$\operatorname{Im}\rho(th+s)$. Over $\C$, the commuting operators $\rho(h)$ and $\rho(s)$ are
simultaneously diagonalizable. Their joint eigenvalues on the two complex spatial strings are
\[
(m-2k,+1),\qquad (m-2k,-1),
\qquad 0\le k\le m.
\]
Because $m$ is odd, $m-2k\ne0$. For each joint eigenspace choose
$t=\mp(m-2k)^{-1}$ so that $th+s$ vanishes on that eigenspace. Since $th+s$ is diagonalizable,
its image does not contain that eigenspace. Hence the corresponding component of $w$ is zero.
All components vanish, so $w=0$ and $\Delta(s)=bz=b\eta_M(s)$.
Restoring the subtracted $a\ad s$ proves (ii).
\end{proof}

\begin{theorem}\label{thm:local-mass}
Every local derivation on the mass central extension is a derivation. More precisely,
\begin{equation}\label{eq:locder-mass}
\LocDer\widehat{\g}^{\,M}_\ell(d)
=\Der\widehat{\g}^{\,M}_\ell(d),
\qquad \ell\in\mathbb N-\frac12.
\end{equation}
\end{theorem}
\begin{proof}
Let $\Delta$ be a local derivation and apply Lemma~\ref{lem:mass-normalization}. Thus, after
subtracting a global derivation, we may suppose \eqref{eq:mass-normalized}.
Lemma~\ref{lem:mass-one-string} reduces the action on the Heisenberg module to cross-spatial
same-level terms, and Lemma~\ref{lem:mass-spatial} eliminates these terms for $d\ge3$.
Lemma~\ref{lem:mass-rotation-images} then gives $\Delta=0$.

For $d=1$, the normalized map already vanishes by the one-string argument in
Lemma~\ref{lem:mass-one-string}. For $d=2$, Lemma~\ref{lem:mass-rotation-images} yields
$\Delta=a\ad s+b\eta_M$, which is a derivation. Restoring the derivation subtracted in the
normalization proves \eqref{eq:locder-mass}.
\end{proof}

\begin{remark}\label{rem:mass-technique}
The mass proof separates into two independent rigidity mechanisms. The one-string projection uses the
known local rigidity of the $\ell$-conformal Galilei algebra. The genuinely new spatial step uses paired
vectors such as
\[
p_{r,k}+\mathrm i p_{t,k}
\quad\text{and}\quad
p_{r,k}+p_{r,k+1},
\]
which force the remaining coefficients to form one fixed skew-symmetric matrix. The mixed
rotation--Heisenberg test vectors then show that this matrix is compatible with the full rotation algebra
only in the $d=2$ central direction.
\end{remark}

\subsection{Local derivations on the exotic central extension}\label{subsec:local-exotic}
We now analyze the exotic extension using the derivation normal form obtained above. The alternating central pairing of the $p$- and $q$-strings provides the rigidity mechanism, while the adapted complex basis allows the local coefficients to be compared level by level.

By Theorem~\ref{thm:exoticder}, for every
$x\in\widehat{\g}^{\,\Theta}_\ell(2)$ there are
$a_x\in\widehat{\g}^{\,\Theta}_\ell(2)$ and
$\lambda_x,\mu_x\in\C$ such that
\begin{equation}\label{eq:localTheta}
\Delta(x)=[a_x,x]+\lambda_x\delta_\Theta(x)+\mu_x\eta_\Theta(x).
\end{equation}
Since the central vector $z$ does not contribute to an inner derivation, we may always write
\begin{equation}\label{eq:pointwise-exotic}
a_x=b_e e+b_f f+b_hh+b_ss+
\sum_{r=0}^{2\ell}b_{p_r}p_r+
\sum_{r=0}^{2\ell}b_{q_r}q_r.
\end{equation}
All coefficients in \eqref{eq:pointwise-exotic}, as well as $\lambda_x$ and $\mu_x$, are allowed to depend on the point $x$.

\begin{lemma}\label{lem:localTheta-support}
Let $\Delta\in\LocDer\widehat{\g}^{\,\Theta}_\ell(2)$. Then
\begin{align*}
\Delta(e)&\in\operatorname{span}\{e,h,p_r,q_r:0\le r\le2\ell-1\},\\
\Delta(f)&\in\operatorname{span}\{f,h,p_r,q_r:1\le r\le2\ell\},\\
\Delta(h)&\in\operatorname{span}\{e,f,p_r,q_r:r\neq\ell\},\\
\Delta(s)&\in\operatorname{span}\{p_r,q_r,z:0\le r\le2\ell\},\\
\Delta(z)&\in\C z,
\end{align*}
and, for every $0\le k\le2\ell$,
\begin{equation}\label{eq:localTheta-support-pq}
\Delta(p_k)\in\operatorname{span}\{p_{k-1},p_k,p_{k+1},z\},\qquad
\Delta(q_k)\in\operatorname{span}\{q_{k-1},q_k,q_{k+1},z\}.
\end{equation}
Here the terms with indices $-1$ and $2\ell+1$ are omitted.
\end{lemma}
\begin{proof}
Insert the pointwise form \eqref{eq:localTheta}--\eqref{eq:pointwise-exotic} into the multiplication table \eqref{eq:exotic-table}. The elements $e$ and $f$ shift the $p$- and $q$-levels by one, while $h$ and $s$ preserve the level. A complementary $q_{2\ell-k}$-component of $a_x$ can add only a central multiple of $z$ to $\Delta(p_k)$, and similarly a complementary $p_{2\ell-k}$-component can add only $z$ to $\Delta(q_k)$. Finally, $z$ is central, $\delta_\Theta(z)=2z$, and $\eta_\Theta$ is non-zero only at $s$.
\end{proof}

\begin{lemma}\label{lem:ez-zero}
If $\Delta$ is a local derivation and $\Delta(e+z)=0$, then
\[
\Delta(e)=\Delta(z)=0.
\]
\end{lemma}
\begin{proof}
By Lemma~\ref{lem:localTheta-support}, the vectors $\Delta(e)$ and $\Delta(z)$ belong to complementary spans of the fixed basis: the first has no $z$-component and the second is contained in $\C z$. Since $\Delta$ is linear,
$0=\Delta(e+z)=\Delta(e)+\Delta(z)$, so both terms vanish.
\end{proof}

\begin{lemma}\label{lem:e-zero-exotic}
Let $\Delta$ be a local derivation with $\Delta(e)=0$. Then:
\begin{enumerate}[label=\textup{(\roman*)}]
\item for $0\le k\le2\ell-1$,
\[
\Delta(p_k)\in\operatorname{span}\{p_{k-1},p_k,p_{k+1}\},\qquad
\Delta(q_k)\in\operatorname{span}\{q_{k-1},q_k,q_{k+1}\};
\]
\item
\[
\Delta(h)\in\operatorname{span}\{e,f,p_0,q_0\};
\]
\item there exists $v\in\widehat{\g}^{\,\Theta}_\ell(2)$ such that
$[v,e]=0$ and $\Delta(f)=[v,f]$.
\end{enumerate}
\end{lemma}
\begin{proof}
We give the coefficient argument because it will be used repeatedly. For $k<2\ell$, write
\[
\Delta(p_k)=A_{k,-}p_{k-1}+A_{k,0}p_k+A_{k,+}p_{k+1}+A_{k,z}z.
\]
Apply locality to $e+p_k$. In the pointwise derivation at this element, the same coefficient $b_{q_{2\ell-k}}$ produces a $q_{2\ell-k-1}$-term through $[q_{2\ell-k},e]$ and a $z$-term through $[q_{2\ell-k},p_k]$. The left-hand side has no $q$-component because $\Delta(e)=0$. Hence $b_{q_{2\ell-k}}=0$, and consequently $A_{k,z}=0$. The proof for $q_k$ is identical.

Next compare the two expressions for $\Delta(h+p_k)$ and $\Delta(h+q_k)$, with $k<2\ell$ and $k\neq\ell$. The complementary $q_{2\ell-k}$ (respectively $p_{2\ell-k}$) coefficient is again tied to a central term. Since the already normalized images of $p_k,q_k$ have no central component, all corresponding coefficients of $\Delta(h)$ vanish except possibly those of $p_0,q_0$. This proves (ii).

Finally fix $r\ge2$ and choose $k=2\ell-r+1$, so that $1\le k\le2\ell-1$.
In the locality relation for $f+p_k$, the coefficient of $p_{r-1}$ in the pointwise inner element
produces both the $p_r$-component through $[p_{r-1},f]$ and a central component through
$[p_{r-1},p_k]$.  The central component of $\Delta(p_k)$ has already been shown to vanish, so
that coefficient is zero; hence the $p_r$-coefficient of $\Delta(f)$ is zero.  The same argument
with $f+q_k$ eliminates every $q_r$-coefficient for $r\ge2$.  Thus
\[
\Delta(f)=\gamma f+Ah+Bp_1+Cq_1.
\]
At $e+f$, the pointwise $h$-coefficient would create equal and opposite $e$- and $f$-components.
Since $\Delta(e)=0$ and the left-hand side has no $e$-component, that coefficient is zero, and
therefore $\gamma=0$. Therefore, for
\[
v=Ae-\frac{B}{2\ell}p_0-\frac{C}{2\ell}q_0
\]
one has $[v,e]=0$ and, by \eqref{eq:exotic-table}, $[v,f]=\Delta(f)$.
\end{proof}

\begin{lemma}\label{lem:f-zero-exotic}
Let $\Delta$ be a local derivation with $\Delta(f)=0$. Then, for $1\le k\le2\ell$,
\[
\Delta(p_k)\in\operatorname{span}\{p_{k-1},p_k,p_{k+1}\},\qquad
\Delta(q_k)\in\operatorname{span}\{q_{k-1},q_k,q_{k+1}\},
\]
and
\[
\Delta(h)\in\operatorname{span}\{e,f,p_{2\ell},q_{2\ell}\}.
\]
\end{lemma}
\begin{proof}
We record the symmetric coefficient argument explicitly. For $k\ge1$, write
\[
\Delta(p_k)=A_{k,-}p_{k-1}+A_{k,0}p_k+A_{k,+}p_{k+1}+A_{k,z}z.
\]
Apply locality to $f+p_k$. In the pointwise inner element, the coefficient of
$q_{2\ell-k}$ contributes simultaneously to the component $q_{2\ell-k+1}$ through
$[q_{2\ell-k},f]$ and to the central component through
$[q_{2\ell-k},p_k]$. Since $\Delta(f)=0$ and the left-hand side has no $q$-component,
that coefficient must vanish; hence $A_{k,z}=0$. The same argument with $f+q_k$
eliminates the central term in $\Delta(q_k)$.

Next compare locality at $h+p_k$ and $h+q_k$, now with $k\ge1$ and $k\ne\ell$.
The complementary Heisenberg coefficient is again tied to a central term, already known to be absent
from $\Delta(p_k)$ and $\Delta(q_k)$. Therefore all $p_r,q_r$-coefficients of $\Delta(h)$ vanish
except possibly those at the top level $r=2\ell$. Together with
Lemma~\ref{lem:localTheta-support} this gives
\[
\Delta(h)\in\operatorname{span}\{e,f,p_{2\ell},q_{2\ell}\}.
\]
This proves the lemma.
\end{proof}

\begin{lemma}\label{lem:ef-zero-exotic}
If $\Delta(e)=\Delta(f)=0$, then
\[
\Delta(h)=0,
\]
and there exist scalars $a_k,b_k\in\C$ such that
\begin{equation}\label{eq:diagonal-pq}
\Delta(p_k)=a_kp_k,\qquad \Delta(q_k)=b_kq_k,
\qquad 0\le k\le2\ell.
\end{equation}
\end{lemma}
\begin{proof}
Lemmas~\ref{lem:e-zero-exotic} and \ref{lem:f-zero-exotic} first imply
$\Delta(h)\in\operatorname{span}\{e,f\}$. Write
$\Delta(h)=Ae+Bf$. For arbitrary $x,y\in\C$, locality at $h+xe+yf$ gives, after comparing the $e,f,h$ coefficients,
\[
A=2xb_h-2b_e,\qquad B=2b_f-2yb_h,\qquad yb_e-xb_f=0.
\]
Taking $(x,y)=(0,1)$ and then $(1,1)$ yields $A=B=0$.

It remains to remove the neighboring levels in the images of the $p$- and $q$-strings.
From the previous two lemmas each image has at most three adjacent levels.  For example, write
\[
\Delta(p_1)=u_0p_0+u_1p_1+u_2p_2.
\]
At the point $f+p_1$, the $h$-coefficient of the pointwise expression is the same coefficient
$b_e$ which produces the $p_0$-term through $[e,p_1]=p_0$.  Since the left-hand side has no
$h$-component, $b_e=0$, hence $u_0=0$.  Similarly, locality at $e+p_{2\ell-1}$ gives that the
$p_{2\ell}$-coefficient of $\Delta(p_{2\ell-1})$ is zero.  The same argument applies to the
$q$-string.

Now use the mixed anchor vectors
\[
p_1+q_{2\ell-1},\qquad q_1+p_{2\ell-1}.
\]
When $\ell\ge2$, these are distinct anchor vectors.  In the first one, the coefficient $b_e$ is
forced to vanish by the $p_0$-component and the coefficient $b_f$ is forced to vanish by the
$q_{2\ell}$-component.  Consequently the remaining $p_2$- and $q_{2\ell-2}$-coefficients vanish.
The second mixed vector gives the analogous conclusion for $q_2$ and $p_{2\ell-2}$.  Hence the
four anchor images are diagonal.

There is a minor endpoint coincidence when $\ell=1$, because $2\ell-1=1$.  In that case the
preceding tests $f+p_1$, $e+p_1$, $f+q_1$, and $e+q_1$ already show that both
$\Delta(p_1)$ and $\Delta(q_1)$ are diagonal.  The only remaining off-diagonal terms can occur in
$\Delta(p_0),\Delta(p_2),\Delta(q_0),\Delta(q_2)$.  Locality at
\[
q_1+p_0,\qquad q_1+p_2,\qquad p_1+q_0,\qquad p_1+q_2
\]
then removes these four terms: the diagonal anchor image forces, respectively, the relevant $f$- or
$e$-coefficient of the pointwise inner element to be zero.  Thus \eqref{eq:diagonal-pq} holds for
$\ell=1$.

Assume henceforth $\ell\ge2$.  For arbitrary $k$, compare
\[
q_1+p_k,\quad q_{2\ell-1}+p_k,
\qquad
p_1+q_k,\quad p_{2\ell-1}+q_k.
\]
Because the anchor images are already diagonal, the first and third test vectors force the
$f$-coefficient of the attached pointwise inner element to vanish whenever an upper neighboring
term is possible, while the second and fourth force the $e$-coefficient to vanish whenever a lower
neighboring term is possible.  Therefore no $p_{k+1},q_{k+1}$ or $p_{k-1},q_{k-1}$ term can occur.
This proves \eqref{eq:diagonal-pq} for every $k$.
\end{proof}

\begin{lemma}\label{lem:s-image-exotic}
If $\Delta(e)=\Delta(f)=\Delta(h)=0$, then
\[
\Delta(s)=\alpha z
\]
for some $\alpha\in\C$.
\end{lemma}
\begin{proof}
Write
\[
\Delta(s)=\sum_{i=0}^{2\ell}A_ip_i+
\sum_{i=0}^{2\ell}B_iq_i+Cz.
\]
Apply locality to $xh+s$ with arbitrary $x\in\C$. Comparing the $p_i$- and $q_i$-coefficients shows immediately that
$A_i=B_i=0$ for $i\neq\ell$. Thus
\[
\Delta(s)=A_\ell p_\ell+B_\ell q_\ell+Cz.
\]
To remove the two middle-weight coefficients, use the single test element
\[
\xi=s+e+\frac{1}{4\ell^2}f.
\]
The $p$-part of $[a_\xi,\xi]$ has coefficients satisfying, away from $i=\ell$, the recurrence
\begin{equation}\label{eq:rec-p}
b_{p_i}+(i+1)b_{p_{i+1}}+
\frac{2\ell-i+1}{4\ell^2}b_{p_{i-1}}=0,
\end{equation}
with $b_{p_{-1}}=b_{p_{2\ell+1}}=0$; at $i=\ell$ the same left-hand side equals $-A_\ell$. Iterating \eqref{eq:rec-p} from both endpoints gives
\[
b_{p_{\ell-1}}=-\frac{2\ell^2}{\ell+1}b_{p_\ell},
\qquad
b_{p_{\ell+1}}=-\frac{1}{2(\ell+1)}b_{p_\ell}.
\]
Substitution into the middle equation gives $A_\ell=0$. For the $q$-string, if $d_i$ denotes the coefficient of $q_i$ in the pointwise inner element, the corresponding recurrence is
\[
d_i-(i+1)d_{i+1}-\frac{2\ell-i+1}{4\ell^2}d_{i-1}=0
\qquad(i\ne\ell),
\]
with the middle left-hand side equal to $B_\ell$.  After the substitution
$\widetilde d_i=(-1)^i d_i$, this becomes exactly the same homogeneous recurrence as
\eqref{eq:rec-p}; the same endpoint computation therefore gives $B_\ell=0$. Hence only the
central term remains.
\end{proof}

\begin{lemma}\label{lem:scalar-exotic}
Assume
\[
\Delta(e)=\Delta(f)=\Delta(h)=\Delta(s)=\Delta(z)=0.
\]
Then there exists $a\in\C$ such that
\begin{equation}\label{eq:scalar-exotic}
\Delta(p_k)=ap_k,\qquad
\Delta(q_k)=-aq_k
\qquad(0\le k\le2\ell).
\end{equation}
\end{lemma}
\begin{proof}
By Lemma~\ref{lem:ef-zero-exotic}, write
$\Delta(p_k)=a_kp_k$ and $\Delta(q_k)=b_kq_k$.  Put
\[
\bar k=2\ell-k,
\qquad
c_k=(-1)^k(2\ell-k)!k!.
\]
We now keep the relevant pointwise coefficients visible; this also fixes all signs and factorials.

First take $k\ne\ell$ and
\[
x=h+p_k+q_{\bar k}-\frac{c_k}{2(\ell-k)}z.
\]
Let $u_h,u_s,u_{p_k},u_{q_{\bar k}}$ and $\lambda$ denote the corresponding coefficients in the
pointwise derivation.  Comparing the $p_k,q_{\bar k},z$ coefficients gives
\begin{align}
a_k&=u_s+2(\ell-k)u_h-2(\ell-k)u_{p_k}+\lambda,\label{eq:scalar-audit1a}\\
b_{\bar k}&=-u_s-2(\ell-k)u_h+2(\ell-k)u_{q_{\bar k}}+\lambda,\label{eq:scalar-audit1b}\\
0&=c_k\bigl(u_{p_k}-u_{q_{\bar k}}\bigr)
-\frac{c_k}{\ell-k}\lambda.\label{eq:scalar-audit1c}
\end{align}
Equation~\eqref{eq:scalar-audit1c} says
$2(\ell-k)(u_{p_k}-u_{q_{\bar k}})=2\lambda$; adding
\eqref{eq:scalar-audit1a} and \eqref{eq:scalar-audit1b} therefore yields
\begin{equation}\label{eq:rel1}
a_k+b_{\bar k}=0.
\end{equation}

Next let $1\le k\le2\ell$ and use
\[
x=e+p_{k-1}+q_{\bar k}-\frac{c_k}{k}z.
\]
The $e$-coefficient first gives $u_h=0$.  The coefficients of $p_{k-1}$ and $q_{\bar k}$ are
\begin{align*}
a_{k-1}&=u_s-ku_{p_k}+\lambda,\\
b_{\bar k}&=-u_s-(\bar k+1)u_{q_{\bar k+1}}+\lambda.
\end{align*}
The central coefficient is
\[
0=c_ku_{p_k}
+\frac{\bar k+1}{k}c_k u_{q_{\bar k+1}}
-\frac{2c_k}{k}\lambda,
\]
that is,
\[
ku_{p_k}+(\bar k+1)u_{q_{\bar k+1}}=2\lambda.
\]
Adding the preceding two displayed formulas gives
\begin{equation}\label{eq:rel2}
a_{k-1}+b_{\bar k}=0.
\end{equation}

Finally, for $0\le k\le2\ell-1$ take
\[
x=f+p_{k+1}+q_{\bar k}-\frac{c_k}{2\ell-k}z.
\]
The $f$-coefficient forces $u_h=0$.  We obtain
\begin{align*}
a_{k+1}&=u_s-(2\ell-k)u_{p_k}+\lambda,\\
b_{\bar k}&=-u_s-(k+1)u_{q_{\bar k-1}}+\lambda,
\end{align*}
and the central coefficient gives
\[
(2\ell-k)u_{p_k}+(k+1)u_{q_{\bar k-1}}=2\lambda.
\]
Hence
\begin{equation}\label{eq:rel3}
a_{k+1}+b_{\bar k}=0.
\end{equation}

For $0\le k<\ell$, equations~\eqref{eq:rel1} and \eqref{eq:rel3} give
$a_k=a_{k+1}$, while for $\ell<k\le2\ell$ equations~\eqref{eq:rel1} and
\eqref{eq:rel2} give $a_k=a_{k-1}$.  The endpoint choices $k=\ell-1$ in \eqref{eq:rel1},\eqref{eq:rel3} and
$k=\ell+1$ in \eqref{eq:rel1},\eqref{eq:rel2} are included in these two chains, so both
chains meet at $a_\ell$.  Thus
\[
a_0=\cdots=a_{2\ell}=a.
\]
Relations~\eqref{eq:rel1}--\eqref{eq:rel3} then imply
$b_0=\cdots=b_{2\ell}=-a$, proving \eqref{eq:scalar-exotic}.
\end{proof}

\begin{remark}\label{rem:exotic-coeff-audit}
The three central corrections in Lemma~\ref{lem:scalar-exotic} are not ad hoc.  They are exactly
\[
-\frac{c_k}{2(\ell-k)},\qquad
-\frac{c_k}{k},\qquad
-\frac{c_k}{2\ell-k},
\]
so that the central equation cancels the scalar outer contribution $2\lambda z$.  This is the
mechanism that turns the pointwise diagonal coefficients into one global inner derivation $a\ad s$.
\end{remark}

\begin{theorem}\label{thm:local-exotic}
Every local derivation on the exotic central extension
$\widehat{\g}^{\,\Theta}_\ell(2)$ is a derivation. Equivalently,
\begin{equation}\label{eq:locder-exotic}
\LocDer\widehat{\g}^{\,\Theta}_\ell(2)
=
\Der\widehat{\g}^{\,\Theta}_\ell(2).
\end{equation}
\end{theorem}
\begin{proof}
Let $\Delta$ be a local derivation. Choose
$D_0\in\Der\widehat{\g}^{\,\Theta}_\ell(2)$ such that
$\Delta(e+z)=D_0(e+z)$, and put
$\Delta_1=\Delta-D_0$. Then $\Delta_1$ is local and
$\Delta_1(e+z)=0$. By Lemma~\ref{lem:ez-zero},
\[
\Delta_1(e)=\Delta_1(z)=0.
\]
Lemma~\ref{lem:e-zero-exotic} provides $v$ with
$[v,e]=0$ and $\Delta_1(f)=[v,f]$. Hence
$\Delta_2=\Delta_1-\ad v$ satisfies
\[
\Delta_2(e)=\Delta_2(f)=\Delta_2(z)=0.
\]
Lemma~\ref{lem:ef-zero-exotic} gives $\Delta_2(h)=0$, and Lemma~\ref{lem:s-image-exotic} gives
$\Delta_2(s)=\alpha z=\alpha\eta_\Theta(s)$. Set
$\Delta_3=\Delta_2-\alpha\eta_\Theta$. Then
\[
\Delta_3(e)=\Delta_3(f)=\Delta_3(h)=\Delta_3(s)=\Delta_3(z)=0.
\]
By Lemma~\ref{lem:scalar-exotic},
$\Delta_3(p_k)=ap_k$ and $\Delta_3(q_k)=-aq_k$ for all $k$. Since these are precisely the values of $a\,\ad s$, we have
$\Delta_3=a\,\ad s$. Therefore
\[
\Delta=D_0+\ad v+\alpha\eta_\Theta+a\,\ad s
\]
is a derivation.
\end{proof}

\begin{remark}\label{rem:exotic-selfcontained}
The proof of Theorem~\ref{thm:local-exotic} uses only the multiplication table \eqref{eq:exotic-table} and the derivation decomposition established in Theorem~\ref{thm:exoticder}. In particular, the exotic case is self-contained in the present manuscript.
\end{remark}

\subsection{Comparison of the three cases}\label{subsec:comparison}
We finish the local-derivation section by comparing the three outcomes. The comparison makes clear that the decisive distinction is not merely the presence of a radical, but whether that radical is abelian or is tied together by a nondegenerate central Heisenberg-type pairing.

The local-derivation problem is now complete for all three Lie algebras considered in this paper.  The two central extensions are rigid: Theorems~\ref{thm:local-mass} and \ref{thm:local-exotic} give
\[
\LocDer\widehat{\g}^{\,M}_\ell(d)=\Der\widehat{\g}^{\,M}_\ell(d),
\qquad
\LocDer\widehat{\g}^{\,\Theta}_\ell(2)=\Der\widehat{\g}^{\,\Theta}_\ell(2).
\]
For the algebra without central extension, Theorem~\ref{thm:local-noncentral} shows that the same equality holds in every spatial dimension $d\ge2$ and also in dimension $d=1$ except when $2\ell\ge3$ is odd.  In that exceptional one-dimensional case there is an additional irreducible pure-local component isomorphic to $U_{2\ell-2}$ and of dimension $2\ell-1$.

This comparison isolates the role of the central cocycle.  The abelian radical permits range-local coboundaries which need not be genuine coboundaries, whereas the Heisenberg pairing in the mass and exotic extensions couples complementary weight spaces and removes that freedom.  In particular, the mass extension converts the exceptional odd-$2\ell$ non-central cases into rigid local-derivation algebras.

\section{Dimensions and Lie algebra structure of local derivations}\label{sec:locder-structure}
Beyond the classification of individual local derivations, it is natural to ask what algebraic structure is carried by the whole space $\LocDer$.  A recent preprint of Hao and Chen gives a general pointwise-invariant-set approach to the commutator closure of $\LocDer$~\cite{HC}; the purpose here is more explicit.  We compute the dimensions in all three families and determine the concrete Lie algebra structure, including the module action, radical, derived algebra and center in the exceptional one-dimensional non-central case.

Put
\[
m=2\ell,
\qquad
r_d=\dim\mathfrak{so}_d=\frac{d(d-1)}2.
\]

\begin{proposition}\label{prop:dimension-formulas}
The following dimension formulas hold.
\begin{enumerate}[label=\textup{(\roman*)}]
\item For the algebra without central extension,
\[
\dim\g_\ell(d)=3+r_d+d(m+1),
\qquad
\dim\Der\g_\ell(d)=4+r_d+d(m+1).
\]
If $d\ge2$, or if $d=1$ with $m$ even or $m=1$, then the same formula is the dimension of $\LocDer\g_\ell(d)$.  If $d=1$ and $m\ge3$ is odd, then
\[
\dim\LocDer\g_\ell(1)=2m+4=4\ell+4.
\]
\item For the mass central extension,
\[
\dim\widehat{\g}^{\,M}_\ell(d)=4+r_d+d(m+1).
\]
Moreover,
\[
\dim\LocDer\widehat{\g}^{\,M}_\ell(d)
=\dim\Der\widehat{\g}^{\,M}_\ell(d)
=
\begin{cases}
4+r_d+d(m+1),& d\ne2,\\[1mm]
2m+8,&d=2.
\end{cases}
\]
\item For the exotic central extension,
\[
\dim\widehat{\g}^{\,\Theta}_\ell(2)=4\ell+7,
\qquad
\dim\LocDer\widehat{\g}^{\,\Theta}_\ell(2)
=\dim\Der\widehat{\g}^{\,\Theta}_\ell(2)=4\ell+8.
\]
\end{enumerate}
\end{proposition}
\begin{proof}
The non-central algebra has zero center, so
$\dim\Inn\g_\ell(d)=\dim\g_\ell(d)$.  Theorem~\ref{thm:der0} adds the one-dimensional outer direction $\C\delta_0$.  In the exceptional case of Theorem~\ref{thm:local-noncentral}, the additional pure-local summand has dimension $m-1$, and therefore
\[
\dim\LocDer\g_\ell(1)
=(m+5)+(m-1)=2m+4.
\]

Both central extensions have one-dimensional center, so their inner derivation algebras have codimension one in the underlying Lie algebras.  Theorem~\ref{thm:massder} adds one outer direction for $d\ne2$ and two for $d=2$, while Theorem~\ref{thm:exoticder} adds two outer directions in the exotic case.  The equalities $\LocDer=\Der$ from Theorems~\ref{thm:local-mass} and~\ref{thm:local-exotic} then give the stated formulas.
\end{proof}

The exceptional non-central case has an especially transparent structure.  Recall that the pure-local component $\mathcal P_m\cong U_{m-2}$ consists of maps which vanish on $U_m$ and map $\mathfrak{sl}_2$ into $U_m$.

\begin{theorem}\label{thm:locder-lie-structure}
Let $d=1$ and let $m=2\ell\ge3$ be odd.  Then $\LocDer\g_\ell(1)$ is a Lie algebra under the commutator of endomorphisms.  More precisely,
\begin{equation}\label{eq:locder-semidirect}
\LocDer\g_\ell(1)
=\Der\g_\ell(1)\ltimes\mathcal P_m,
\qquad
[\mathcal P_m,\mathcal P_m]=0,
\end{equation}
where $\mathcal P_m\cong U_{m-2}$ is an abelian ideal.  Equivalently,
\begin{equation}\label{eq:locder-abstract-structure}
\LocDer\g_\ell(1)
\cong
\bigl(\mathfrak{sl}_2\oplus\C d_0\bigr)
\ltimes\bigl(U_m\oplus U_{m-2}\bigr),
\end{equation}
where $d_0$ commutes with $\mathfrak{sl}_2$ and acts as the identity on both $U_m$ and $U_{m-2}$.
\end{theorem}
\begin{proof}
Let $F,G\in\mathcal P_m$.  Since both maps vanish on $U_m$ and map $\mathfrak{sl}_2$ into $U_m$, one has
\[
F\circ G=G\circ F=0.
\]
Thus $[F,G]=0$.

We next show that $\mathcal P_m$ is invariant under $\Der\g_\ell(1)$.  By Theorem~\ref{thm:der0}, every derivation is a sum of an inner derivation and a scalar multiple of $\delta_0$.  If $a\in\mathfrak{sl}_2$, then for $F\in\mathcal P_m$ and $x\in\mathfrak{sl}_2$,
\[
[\ad a,F](x)=\rho_m(a)F(x)-F([a,x]),
\]
which is precisely the natural $\mathfrak{sl}_2$-action on
$\operatorname{Hom}(\mathfrak{sl}_2,U_m)$.  The summand
$\mathcal P_m\cong U_{m-2}$ is an $\mathfrak{sl}_2$-submodule by Proposition~\ref{prop:sl2-local-coboundaries}, so
$[\ad a,F]\in\mathcal P_m$.

If $u\in U_m$, then
\[
[\ad u,F]=0.
\]
Indeed, $[u,U_m]=0$, while $[u,\mathfrak{sl}_2]\subset U_m$ and $F(U_m)=0$.  Finally,
\[
[\delta_0,F]=F,
\]
because $\delta_0$ is the identity on $U_m$ and vanishes on $\mathfrak{sl}_2$.  Hence $\mathcal P_m$ is an abelian ideal of the full local-derivation space, proving \eqref{eq:locder-semidirect}.

Since $Z(\g_\ell(1))=0$, the inner derivations identify with
$\mathfrak{sl}_2\ltimes U_m$.  Moreover,
\[
[\delta_0,\ad u]=\ad(\delta_0u)=\ad u,
\qquad
[\delta_0,\ad a]=0
\]
for $u\in U_m$ and $a\in\mathfrak{sl}_2$.  Together with the action on $\mathcal P_m$ computed above, this gives the abstract semidirect-product description \eqref{eq:locder-abstract-structure}.
\end{proof}

\begin{corollary}\label{cor:locder-radical}
Under the assumptions of Theorem~\ref{thm:locder-lie-structure}, the Levi factor of
$\LocDer\g_\ell(1)$ is $\mathfrak{sl}_2$, and
\[
\operatorname{Rad}(\LocDer\g_\ell(1))
=\C d_0\ltimes(U_m\oplus U_{m-2}).
\]
Furthermore,
\[
[\LocDer\g_\ell(1),\LocDer\g_\ell(1)]
=\mathfrak{sl}_2\ltimes(U_m\oplus U_{m-2}),
\qquad
Z(\LocDer\g_\ell(1))=0.
\]
\end{corollary}
\begin{proof}
All assertions follow directly from \eqref{eq:locder-abstract-structure}.  The two irreducible modules are non-trivial, $d_0$ acts non-trivially on both of them, and $\mathfrak{sl}_2$ is semisimple.  Hence the displayed solvable ideal is the radical, the commutator has codimension one, and no non-zero element can be central.
\end{proof}

\begin{remark}\label{rem:all-locder-lie}
Combining Theorem~\ref{thm:locder-lie-structure} with the rigidity theorems shows that in every case considered in this paper the space of local derivations is a Lie algebra.  In the rigid cases this is immediate from $\LocDer=\Der$; in the exceptional family the additional local maps form the abelian ideal $\mathcal P_m$.
\end{remark}

\section{Conclusion}
We conclude by collecting the principal consequences of the preceding derivation and local-derivation classifications and by emphasizing the structural effect of the central cocycles.

We have obtained a uniform local-derivation picture for the conformal Galilei algebra and its two standard one-dimensional central extensions. The derivation calculations show that the scalar outer derivation persists in all three settings, with the central element acquiring twice the module weight in the two central extensions; in spatial dimension two those extensions also admit an additional central shear. The local analysis then separates sharply according to the radical.

For the abelian radical, the problem reduces to algebraic reflexivity of the module action and to the classification of range-local coboundaries. This produces the exceptional irreducible component $U_{2\ell-2}$ exactly when $d=1$ and $2\ell\ge3$ is odd. By contrast, the mass and exotic Heisenberg pairings couple complementary levels and eliminate this extra freedom, so every local derivation is a derivation. This central-cocycle rigidification is the structural distinction linking the three cases.  In addition, the exceptional non-central local-derivation space is itself a Lie algebra: it is the semidirect product of the ordinary derivation algebra with an abelian irreducible pure-local ideal.  Thus the classification determines not only which local derivations occur, but also the algebraic structure carried by the full local-derivation space.

\section*{Declarations}
\noindent\textbf{Conflict of interest.}
The authors declare that they have no relevant financial or non-financial interests to disclose.

\medskip
\noindent\textbf{Funding.}
The authors declare that no funds, grants, or other financial support were received for the preparation of this manuscript.

\medskip
\noindent\textbf{Data availability.}
Data sharing is not applicable to this article because no datasets were generated or analyzed during the current study.

\medskip
\noindent\textbf{Author contributions.}
K.A., B.B.S., and B.B.Y. contributed to the conceptualization, mathematical analysis, verification, and preparation of the manuscript. All authors read and approved the final manuscript.


\begin{thebibliography}{99}

\bibitem{AAYAssoc}
K.~Abdurasulov, Sh.~A.~Ayupov and B.~B.~Yusupov,
\emph{Local and 2-local derivations on filiform associative algebras},
J. Algebra Appl. \textbf{24} (2025), no.~10, 2550242,
DOI: 10.1142/S0219498825502421.

\bibitem{AdYQF}
J.~Q.~Adashev and B.~B.~Yusupov,
\emph{Local derivation of naturally graded quasi-filiform Leibniz algebras},
Uzbek Math. J. \textbf{65} (2021), no.~1, 28--41,
DOI: 10.29229/uzmj.2021-1-3.

\bibitem{AdYNull}
J.~Q.~Adashev and B.~B.~Yusupov,
\emph{Local derivations and automorphisms of direct sum null-filiform Leibniz algebras},
Lobachevskii J. Math. \textbf{43} (2022), no.~12, 3407--3413,
DOI: 10.1134/S1995080222150021.

\bibitem{AizawaIsaac}
N.~Aizawa and P.~S.~Isaac,
\emph{On irreducible representations of the exotic conformal Galilei algebra},
J. Phys. A: Math. Theor. \textbf{44} (2011), 035401,
DOI: 10.1088/1751-8113/44/3/035401.

\bibitem{AizawaIsaacKimura}
N.~Aizawa, P.~S.~Isaac and Y.~Kimura,
\emph{Highest weight representations and Kac determinants for a class of conformal Galilei algebras with central extension},
Int. J. Math. \textbf{23} (2012), no.~11, 1250118,
DOI: 10.1142/S0129167X12501182.

\bibitem{ANYLm}
A.~K.~Alauadinov, Z.~Normatov and B.~B.~Yusupov,
\emph{2-local derivations on the perfect Lie algebras $\mathcal L^m$},
Algebraic Structures and Their Applications \textbf{12} (2025), no.~3, 263--271.

\bibitem{AYCG}
A.~K.~Alauadinov and B.~B.~Yusupov,
\emph{Local derivations of conformal Galilei algebra},
Commun. Algebra \textbf{52} (2024), no.~6, 2489--2508,
DOI: 10.1080/00927872.2023.2301539.

\bibitem{AYSch}
A.~K.~Alauadinov and B.~B.~Yusupov,
\emph{Local derivations of the Schr\"odinger algebras},
Algebra Colloq. \textbf{32} (2025), no.~3, 467--480,
DOI: 10.1142/S1005386725000343.

\bibitem{AYSchn}
A.~K.~Alauadinov and B.~B.~Yusupov,
\emph{Local derivation on the Schr\"odinger Lie algebra in $(n+1)$-dimensional space-time},
Bull. Karaganda Univ. Math. Ser. (2026), no.~1(121), 55--66,
DOI: 10.31489/2026M1/55-66.

\bibitem{Andrzejewski}
K.~Andrzejewski, J.~Gonera and P.~Ma\'{s}lanka,
\emph{Nonrelativistic conformal groups and their dynamical realizations},
Phys. Rev. D \textbf{86} (2012), 065009,
DOI: 10.1103/PhysRevD.86.065009.

\bibitem{AAtYAnti}
Sh.~A.~Ayupov, Kh.~Atajonov and B.~B.~Yusupov,
\emph{Local and 2-local anti-derivations on solvable Lie algebras},
European J. Math. \textbf{11} (2025), no.~3, Article~52,
DOI: 10.1007/s40879-025-00841-w.

\bibitem{AKHY}
Sh.~A.~Ayupov, A.~Kh.~Khudoyberdiyev and B.~B.~Yusupov,
\emph{Local and 2-local derivations of solvable Leibniz algebras},
Int. J. Algebra Comput. \textbf{30} (2020), no.~6, 1185--1197,
DOI: 10.1142/S021819672050037X.

\bibitem{AK}
Sh.~A.~Ayupov and K.~K.~Kudaybergenov,
\emph{Local derivations on finite-dimensional Lie algebras},
Linear Algebra Appl. \textbf{493} (2016), 381--398,
DOI: 10.1016/j.laa.2015.11.034.

\bibitem{AKR}
Sh.~A.~Ayupov, K.~K.~Kudaybergenov and I.~S.~Rakhimov,
\emph{2-local derivations on finite-dimensional Lie algebras},
Linear Algebra Appl. \textbf{474} (2015), 1--11,
DOI: 10.1016/j.laa.2015.01.016.

\bibitem{AKYp}
Sh.~A.~Ayupov, K.~K.~Kudaybergenov and B.~B.~Yusupov,
\emph{Local and 2-local derivations of $p$-filiform Leibniz algebras},
J. Math. Sci. \textbf{245} (2020), no.~3, 359--367,
DOI: 10.1007/s10958-020-04697-1.

\bibitem{AKYWitt}
Sh.~A.~Ayupov, K.~K.~Kudaybergenov and B.~B.~Yusupov,
\emph{2-local derivations on generalized Witt algebras},
Linear Multilinear Algebra \textbf{69} (2021), no.~16, 3130--3140,
DOI: 10.1080/03081087.2019.1708846.

\bibitem{AKYLocSimple}
Sh.~A.~Ayupov, K.~K.~Kudaybergenov and B.~B.~Yusupov,
\emph{Local and 2-local derivations of locally simple Lie algebras},
J. Math. Sci. \textbf{278} (2024), no.~4, 613--622,
DOI: 10.1007/s10958-024-06943-2.

\bibitem{AYQF}
Sh.~A.~Ayupov and B.~B.~Yusupov,
\emph{Local mapping on naturally graded quasi-filiform Leibniz algebras},
Bull. Inst. Math. \textbf{9} (2026), no.~3, 94--104.

\bibitem{ChenZhaoZhao}
Y.~Chen, K.~Zhao and Y.~Zhao,
\emph{Local derivations on Witt algebras},
Linear Multilinear Algebra \textbf{70} (2022), no.~6, 1159--1172,
DOI: 10.1080/03081087.2020.1754750.

\bibitem{ChenLiYu}
Z.~Chen, M.~Li and Y.~Yu,
\emph{Derivations, $\frac12$-derivations and transposed Poisson algebra structures on conformal Galilei algebras},
Commun. Algebra (2026), published online,
DOI: 10.1080/00927872.2025.2604120.

\bibitem{GalajinskyMasterov}
A.~Galajinsky and I.~Masterov,
\emph{Dynamical realization of $\ell$-conformal Galilei algebra and oscillators},
Nucl. Phys. B \textbf{866} (2013), no.~2, 212--227,
DOI: 10.1016/j.nuclphysb.2012.09.004.

\bibitem{HC}
Z.~Hao and L.~Chen,
\emph{Two conjectures and applications on local and almost inner derivations},
Preprint (2026).

\bibitem{Kadison}
R.~V.~Kadison,
\emph{Local derivations},
J. Algebra \textbf{130} (1990), no.~2, 494--509.

\bibitem{KOK}
K.~K.~Kudaybergenov, B.~A.~Omirov and T.~K.~Kurbanbaev,
\emph{Local derivations on solvable Lie algebras of maximal rank},
Commun. Algebra \textbf{50} (2022), no.~9, 3816--3826,
DOI: 10.1080/00927872.2022.2045494.

\bibitem{LarsonSourour}
D.~R.~Larson and A.~R.~Sourour,
\emph{Local derivations and local automorphisms of $B(X)$},
Proc. Sympos. Pure Math. \textbf{51} (1990), Part~2, 187--194.

\bibitem{RVY}
D.~Reymbaeva, N.~Z.~Vaisova and B.~B.~Yusupov,
\emph{Local super-derivations of the $n$-th super Schr\"odinger algebras},
Bol. Soc. Mat. Mex. \textbf{31} (2025), Article~139.

\bibitem{WangTang}
P.~Wang and X.~Tang,
\emph{2-local derivations on Schr\"odinger algebra in $(2+1)$-dimensional space-time},
Commun. Algebra \textbf{51} (2023), no.~2, 676--687,
DOI: 10.1080/00927872.2022.2108045.

\bibitem{YaoYu}
L.~Yao and Y.~Yu,
\emph{2-local derivations of the $n$-th Schr\"odinger algebra},
Commun. Algebra \textbf{53} (2025), no.~6, 2254--2265,
DOI: 10.1080/00927872.2024.2433682.

\bibitem{YuChenBorel}
Y.~Yu and Z.~Chen,
\emph{Local derivations on Borel subalgebras of finite-dimensional simple Lie algebras},
Commun. Algebra \textbf{48} (2020), no.~1, 1--10,
DOI: 10.1080/00927872.2018.1541465.

\bibitem{YAutp}
B.~B.~Yusupov,
\emph{Local automorphisms of some $p$-filiform Leibniz algebras},
Ricerche Mat. \textbf{74} (2025), 1341--1372.

\bibitem{YYFil}
B.~B.~Yusupov and I.~Yuldashev,
\emph{Local derivations on solvable Lie algebras with a filiform nilradical},
Asian-Eur. J. Math. \textbf{18} (2025), no.~12, 2550061.

\bibitem{ZhaoChengCG}
Y.~Zhao and Y.~Cheng,
\emph{2-local derivation on the conformal Galilei algebra},
arXiv:2103.04237 [math.RA] (2021),
DOI: 10.48550/arXiv.2103.04237.

\end{thebibliography}
\end{document}